\documentclass[12pt]{amsart}
\usepackage[colorlinks=true,citecolor=purple,linkcolor=blue]{hyperref}
\usepackage{amssymb,amscd,amsmath,graphicx,cleveref,xcolor,extarrows}
\usepackage{enumerate, multirow}
\usepackage{fullpage}
\usepackage{comment}
\usepackage{tikz-cd}

\numberwithin{equation}{section}
\newtheorem{theorem}{Theorem}[section]

\newtheorem{corollary}[theorem]{Corollary}
\newtheorem{lemma}[theorem]{Lemma}

\theoremstyle{definition}

\newtheorem{definition}[theorem]{Definition}

\newtheorem{quest}[theorem]{Question}
\newtheorem{prob}[theorem]{Problem}

\theoremstyle{definition}
\newtheorem{example}[theorem]{Example}
\newtheorem{remark}[theorem]{Remark}

\DeclareMathOperator{\Ann}{Ann}
\DeclareMathOperator{\Ass}{Ass}

\DeclareMathOperator{\depth}{depth}

\DeclareMathOperator{\Min}{Min}

\DeclareMathOperator{\reg}{reg}

\DeclareMathOperator{\link}{link}
\DeclareMathOperator{\supp}{supp}
\DeclareMathOperator{\starop}{star}
\DeclareMathOperator{\ord}{ord}

\newcommand{\ZZ}{{\mathbb Z}}
\newcommand{\NN}{{\mathbb N}}
\newcommand{\RR}{{\mathbb R}}

\def\kk{{\Bbbk}}

\def\mm{{\mathfrak m}}
\def\pp{{\mathfrak p}}
\def\qq{{\mathfrak q}}

\def\a{{\bold a}}
\def\b{{\bold b}}
\def\c{{\bold c}}
\def\e{{\bold e}}
\def\0{{\bold 0}}
\def\x{{\bold x}}

\def\F{\mathcal{F}}
\def\G{\mathcal{G}}

\begin{document}
	
	\title{Relative Hochster--Takayama formula and Cohen--Macaulay monomial ideal quotients}
	
	\author{T\`ai Huy H\`a}
	\address{Tulane University \\ Mathematics Department \\
		6823 St. Charles Ave. \\ New Orleans, LA 70118, USA}
	\email{tha@tulane.edu}
	
	\author{Nguyen Cong Minh}
	\address{Faculty of Mathematics and Informatics \\ Hanoi University of Science and Technology \\
		1 Dai Co Viet \\ Hanoi, Vietnam}
	\email{minh.nguyencong@hust.edu.vn}

	\keywords{monomial ideals, quotients, Cohen-Macaulay, Hochster formula, Takayama formula, degree complex, relative simplicial complex}
	\subjclass[2020]{13F55, 55U10}
	
	\begin{abstract}
		Hochster's and Takayama's formulas describes the multigraded components of local cohomology modules of monomial ideals in terms of simplicial complexes. In this paper, we develop a relative version of these formulas for quotients $I/J$ of monomial ideals, expressing the multigraded pieces of local cohomology modules of $I/J$ as reduced relative (co)homology of pairs of degree complexes. As an application, we obtain a relative Reisner criterion characterizing Cohen-Macaulay monomial ideal quotients. We further apply this relative Hochster--Takayama framework to modules arising from symbolic power filtrations, including symbolic quotients $I^{(t)}/I^{(t+1)}$ and symbolic-ordinary discrepancy module $I^{(t)}/I^t$. In particular, for a squarefree monomial ideal $I$, we give a precise classification of when $I^{(t)}/I^{(t+1)}$ is Cohen-Macaulay for all or,  equivalently, for some $t \ge 2$. When $I$ is the edge ideal of a graph, we characterize the Cohen-Macaulayness of $I^{(t)}/I^t$ for all or, equivalently, for some sufficiently large $t$, and analyze the behavior of its dimension function.
	\end{abstract}
	
	\maketitle
	
	
	
\section{Introduction} \label{sec.intro}

The interplay between combinatorial structures and algebraic properties is a recurring theme in the study of monomial ideals. For squarefree monomial ideals $J \subseteq I$, the quotient $I/J$ admits a powerful combinatorial description via the \textit{relative simplicial complex} $(\Delta,\Gamma)$, where $\Delta$ and $\Gamma$ are Stanley-Reisner complexes associated to $I$ and $J$. This connection is established through the \textit{relative Stanley-Reisner theory}, initially developed in \cite{Stanley} and further extended in \cite{AS2016} in their studies of the upper bound conjectures and theorems. 
This theory reveals a deep correspondence: analogous to the familiar Stanley-Reisner correspondence for single squarefree monomial ideals, the Hilbert series of $I/J$ can be directly understood from the $f$-vector (or $h$-vector) of the relative complex $(\Delta,\Gamma)$. Moreover, a relative Hochster's formula exists, providing a method to compute the multi-graded pieces of local cohomology modules of $I/J$.

While relative Stanley-Reisner theory provides a comprehensive framework for quotients of squarefree monomial ideals, a more general approach is needed for arbitrary monomial ideal quotients. Takayama \cite{T2005} introduced \textit{degree complexes} and a generalized Hochster's formula (cf. \cite{Hochster}) to compute the multi-graded pieces of local cohomology modules for any monomial ideal. Our primary goal in this paper is to bridge these two theories to formulate and study \emph{relative degree complexes} associated to an arbitrary monomial ideal quotient $I/J$. 

We will prove a \textit{relative Hochster--Takayama formula} for relative complexes, extending Takayama's work to general monomial ideal quotients. Leaving unexplained terminology until later, our first result is stated as follows.

\medskip

\noindent \textbf{Theorem \ref{HochsterTakayama}.} Let $J\subseteq I$ be monomial ideals in $S = \kk[x_1, \dots, x_n]$. For each $i\in\ZZ$ and $\a\in\ZZ^n$, there is an isomorphism of $\kk$-vector spaces
$$H_{\mm}^{i}(I/J)_{\a}\cong \widetilde{H}^{i-|G_\a|-1}((\Delta_\a(J),\Delta_\a(I)); \kk),$$
and $H_{\mm}^{i}(I/J)_{\a}=(0)$ if $G_\a \not\in \Delta(\sqrt{I})$. (Here, $(\Delta_\a(J),\Delta_\a(I))$ represent the relative degree complex at the multidegree $\a \in \ZZ^n$ of the pair $J \subseteq I$.)
 
\medskip

We also prove a functorial result showing that the isomorphisms of $\kk$-vector spaces in the relative Hochster--Takayama formula are compatible with multiplication by variables on local cohomology modules. In particular, in \Cref{MainLem}, we show that 	for $\a, \b \in \NN^n$, the multiplication by a monomial $\x^\b$ induces a commutative diagram
\[
\begin{tikzcd}
	H^i_{\mathfrak m}(M)_\a
	\arrow[r,"{\cdot \x^\b}"]
	\arrow[d,"\cong"']
	&
	H^i_{\mathfrak m}(M)_{\a+\b}
	\arrow[d,"\cong"]
	\\
	\widetilde H^{\,i-|G_\a|-1}\!
	\big(
	\Delta_\a(J),\,\Delta_\a(I)
	\big)
	\arrow[r]
	&
	\widetilde H^{\,i-|G_\a|-1}\!
	\big(
	\Delta_{\a+\b}(J),\,\Delta_{\a+\b}(I)
	\big),
\end{tikzcd}
\]
where the vertical maps are given by the relative Hochster--Takayama
formula, and the bottom horizontal map is induced by the natural inclusions $\Delta_{\a+\b}(J)\ \subseteq\ \Delta_\a(J) \text{ and }
\Delta_{\a+\b}(I)\ \subseteq\ \Delta_\a(I).$

As a first application of this framework, we establish a \emph{Relative Reisner Criterion}, in \Cref{thm:relative-reisner}, characterizing the Cohen-Macaulayness of monomial ideal quotients in terms of the vanishing of reduced relative homology of links in the degree complexes. This result generalizes the classical Reisner criterion (see \cite{Reisner}) and applies uniformly to a broad class of monomial ideal quotients.
In another application, we apply this theory to study modules arising from \emph{symbolic power filtrations}. Symbolic powers of ideals carry important geometric information but are often harder to understand than ordinary powers. It has been an active research program to understand symbolic powers and their relationship to ordinary powers (cf. \cite{ELS, HHT2007, HH} and references thereafter). For a squarefree monomial ideal $I$, we study both the \emph{symbolic quotients} $I^{(t)}/I^{(t+1)}$ and 
\emph{symbolic-ordinary discrepancy modules} $I^{(t)}/I^t$.

The study of the Cohen-Macaulayness of symbolic powers $I^{(t)}$, for $t \ge 1$, has a rich history. For instance, in \cite{MT2011, V2011}, when $I$ is a squarefree monomial ideal, this property was elegantly characterized by the matroidal nature of the Stanley-Reisner complex associated $I$. We extend this matroidal characterization of  \cite{MT2011, V2011} to show that the Cohen-Macaulayness for some (or equivalently all) symbolic quotient $I^{(t)}/I^{(t+1)}$ forces the same strong combinatorial constraint on the underlying simplicial complex. 
	
	\medskip
	
	\noindent\textbf{Theorem \ref{Main1}.} Let $\Delta$ be a simplicial complex on $[n]$ and let $I=I_{\Delta}$ be its Stanley--Reisner ideal.
	Then the following are equivalent:
	\begin{enumerate}
		\item[(1)] $I^{(t)}/I^{(t+1)}$ is Cohen-Macaulay for all $t\ge1$;
		\item[(2)] $I^{(t)}/I^{(t+1)}$ is Cohen-Macaulay for some $t\ge2$;
		\item[(3)] $\Delta$ is a matroid.
	\end{enumerate}
	Moreover, if $\dim S/I \ge2$, then the above are further equivalent to:
	\begin{enumerate}
		\item[(4)] $I^{(t)}/I^{(t+1)}$ satisfies Serre's condition $(S_2)$ for all $t\ge1$;
		\item[(5)] $I^{(t)}/I^{(t+1)}$ satisfies Serre's condition $(S_2)$ for some $t\ge2$.
	\end{enumerate}
	If $\dim S/I \le1$, then $I^{(t)}/I^{(t+1)}$ satisfies $(S_2)$ trivially for all $t\ge1$.
	
	\medskip
	
	A central theme of \Cref{Main1} is rigidity: the existence
	of a single relative degree complex with disconnected $0$-th homology forces
	depth $1$ and rules out higher Serre conditions, leading to sharp classification
	and collapse results for Cohen-Macaulay, $(S_2)$, and generalized
	Cohen-Macaulay properties; see also \Cref{rmk:gCMandB}.
	
We continue to focus on symbolic-ordinary discrepancy modules $I^{(t)}/I^t$. The function $t \mapsto \dim I^{(t)}/I^t$ is known to be asymptotically constant (see \cite{HPV2008}). We focus on edge ideals of graphs and show that this dimension function is non-decreasing (see Theorem \ref{dim-graph}). For unicyclic and perfect graphs we illustrate precisely where this function stabilizes (see Theorems \ref{thm.unicyclic} and \ref{thm.perfect}). We further give a complete classification for the Cohen-Macaulay property of $I^{(t)}/I^t$ for all or, equivalently, sufficiently large $t$.

\medskip
	
	\noindent\textbf{Theorem \ref{thm.CMEdge}.}  Let $G$ be a graph that is not bipartite and let $I = I(G)$ be its edge ideal. The following are equivalent:
	\begin{enumerate}
		\item $I^{(t)}/I^{t}$ is Cohen-Macaulay for all $t \ge 1$;
		\item $I^{(t)}/I^{t}$ is Cohen-Macaulay for some $t\ge \max\{t(G),c(G)\}+3$;
		\item For all $t \ge 1$, either $I^{(t)}/I^t = (0)$ or $\dim(I^{(t)}/I^t) = 0$;
		\item For any induced odd cycle $C$ in $G$, we have $N[C] = [n]$.
	\end{enumerate}

	\medskip
	
	The study of dimension of Cohen-Macaulayness of symbolic-ordinary discrepancy modules (Theorems \ref{dim-graph}, \ref{thm.unicyclic}, \ref{thm.perfect} and \ref{thm.CMEdge}) is based on investigating modules of the form $S/\sqrt{I:x^\a}$ which, for $\a \in I \setminus J$, is a component of the annihilator of $I/J$.
	
	We remark that while the relative Hochster--Takayama framework robustly captures vanishing properties of local cohomology, and hence Cohen-Macaulay and generalized Cohen-Macaulay conditions, finer properties such as Buchsbaumness or sequentially C-Mness depend on the full multigraded module structure of local cohomology and do not admit comparably concise formulations in terms of degree complexes alone. Accordingly, we have restricted our attention to those properties that can be systematically detected within the functorial relative framework developed in this paper.
	
	\medskip
	
\noindent\textbf{Outlines of the paper.}	In the next section, we collect basic definitions and properties used in the paper. In \Cref{sec.HT}, we provide the Relative Hochster--Takayama formula, \Cref{HochsterTakayama}, and prove a functoriality result, \Cref{MainLem}, for monomial ideal quotients. In \Cref{sec.monIdeals} we restate relative Hochster--Takayama formula for quotients of squarefree monomial ideals, highlighting that, in this case, relative degree complexes are relative complexes of restricted complexes; see Theorem \ref{thm:SR-relative-HT}. We obtain the Relative Reisner's Criterion (\Cref{thm:relative-reisner}), and a criterion for generalized Cohen-Macaulayness (Theorems \ref{thm:gCM-criterion}) in \Cref{sec.gCMB}. \Cref{sec.sym/sym} is devoted to investigating the Cohen-Macaulayness of symbolic quotients of the form $I^{(t)}/I^{(t+1)}$, encapsulating in \Cref{Main1}. The last two sections, \Cref{sec.dim} and \Cref{sec.sym/ord}, focus on symbolic-ordinary discrepancy modules of the form $I^{(t)}/I^t$, when $I$ is the edge ideal of a graph. We prove that the dimension of the symbolic-ordinary discrepancy modules is an increasing function in \Cref{dim-graph}. We classify graphs for which one (equivalently, all) symbolic-ordinary discrepancy module is Cohen-Macaulay in \Cref{thm.CMEdge}.
	
	
	\section{Preliminaries} \label{sec.prel}
	
	This section provides essential definitions and background concepts for the paper. For unexplained terminology, we refer the reader to \cite{BrunsHerzog, MS2005, Stanley}.
	
	\subsection{Simplicial Complexes and Homology}
	
	A \textit{simplicial complex} $\Delta$ on a finite set $[n]=\{1,2,...,n\}$ is formally defined as a collection of subsets of $[n]$ such that if $F\in\Delta$ and $F'\subseteq F$, then $F'\in\Delta$. For convenience in later discussions, the condition that $\{i\}\in\Delta$ for all $i=1,2,...,n$ is not assumed. 
	Set $\dim F=|F|-1$, where $|F|$ means the cardinality of $F$, and $\dim\Delta=\max\{\dim F~\big|~ F\in\Delta\}$, which is called the \emph{dimensio}n of $\Delta$. When $[n]$ is equipped with a linear order, say $<$, $\Delta$ is called an \textit{oriented} simplicial complex. In such a case, we denote $F=\{i_{1},...,i_{r}\}$ for $F\in\Delta$ with the order sequence $i_{1}<...<i_{r}$. 
	
	Let $F$ be a face of a simplicial complex $\Delta$. The \emph{link} of $F$ in $\Delta$, denoted by $\link_\Delta F$ is the simplicial complex whose faces are given by $\{G \subseteq [n] ~\big|~ G \cap F = \varnothing, G \cup F \in \Delta\}$. The \emph{star} center at $F$ in $\Delta$, denoted by $\starop_\Delta F$, is the simplicial complex whose faces are given by $\{\tau \cup \sigma ~\big|~ \tau \subseteq F, \sigma \in \link_\Delta F\}.$
	
	Let $\Delta$ be an oriented simplicial complex with $\dim\Delta=d$. We denote by $\mathcal{C}(\Delta)_{\bullet}$ the \textit{augmented oriented chain complex} of $\Delta$, where 
	$$C_{t}=\bigoplus_{\dim F=t}\mathbb{Z}F \text{ and } \partial F=\sum_{j=0}^{t}(-1)^{j}F_{j}, \text{ for all }  \ F\in\Delta.$$
	(Here, $F_{j}=\{i_{0},\cdots,\hat{i}_{j},\cdots,i_{t}\}$ if $F=\{i_{0},...,i_{t}\}$.)
	
	A \textit{relative simplicial complex} is a pair $(\Delta,\Gamma)$ of simplicial complexes for which $\Gamma\subseteq\Delta$ is a simplicial subcomplex. The faces of $(\Delta,\Gamma)$ are the elements $\{F\in\Delta ~\big|~ F\notin\Gamma\}$. Therefore, an ordinary simplicial complex can be viewed as a relative simplicial complex with $\Gamma=\varnothing$. The \emph{dimension} of a relative simplicial complex $\Phi = (\Delta,\Gamma)$ is defined by $\dim \Phi=\max \{\dim(F) ~\big|~ F\in(\Delta,\Gamma)\}$. If $\Delta$ is an oriented simplex complex and $\Gamma$ is an oriented subcomplex of $\Delta$, then $\mathcal{C}(\Gamma)_{\bullet}$ is a subcomplex of $\mathcal{C}(\Delta)_{\bullet}$, which yields the \textit{quotient complex} $\mathcal{C}(\Delta)_{\bullet}/\mathcal{C}(\Gamma)_{\bullet}$.
	
	The $i$-th \textit{reduced relative simplicial homology module} of $(\Delta, \Gamma)$ is defined as $$\widetilde{H}_{i}((\Delta,\Gamma);\kk)=H_{i}(\mathcal{C}(\Delta)_{\bullet}/\mathcal{C}(\Gamma)_{\bullet}\otimes \kk),$$ and the $i$-th \textit{reduced relative simplicial cohomology module} of $(\Delta, \Gamma)$ is $$\widetilde{H}^{i}((\Delta,\Gamma);\kk)=H^{i}(\operatorname{Hom}_{\mathbb{Z}}(\mathcal{C}(\Delta)_{\bullet}/\mathcal{C}(\Gamma)_{\bullet};\kk)).$$
	For simplicity of notation, we will often omit $\kk$ and write $\widetilde{H}_{i}(\Delta,\Gamma)$ or $\widetilde{H}^{i}(\Delta,\Gamma)$ for the reduced relative simplicial homology and cohomology of a relative simplicial complex $(\Delta,\Gamma)$.
	
	\subsection{Degree Complexes and Symbolic Powers of Monomial Ideals}
	
	Throughout the paper, $\kk$ is an arbitrary field, $S = \kk[x_1, \dots, x_n]$ is a polynomial ring over $\kk$. 	
	The $t$-th \emph{symbolic power} of an ideal $I \subseteq S$ is defined as
	$$I^{(t)} = \bigcap_{\pp \in \Ass(I)} \left(I^tS_\pp \cap S\right).$$
	
	When $I = I_\Delta$ is the Stanley-Reisner ideal associated to a simplicial complex $\Delta$ over $[n]$, its symbolic powers have a better combinatorial understanding. More precisely, let $\F(\Delta)$ denote the set of facets in $\Delta$ and, for $F \in \Delta$, set $P_F = (x_i ~\big|~ i \not\in F).$ In this case, the primary decomposition of $I$ and its symbolic powers are given by
	$$I = \bigcap_{F \in \F(\Delta)} P_F \text{ and } I^{(t)} = \bigcap_{F \in \F(\Delta)} P_F^t.$$
	
	For $\a = (a_1, \cdots, a_n) \in \ZZ^n$, set $G_{\a}=\{i ~\big|~ |a_{i}<0\}$. The \textit{degree complex} $\Delta_{\a}(I)$, of a monomial ideal $I \subseteq S$ at multi-degree $\a$ (as defined in \cite{T2005}), consists of all $F\subseteq[n]$ such that:
	\begin{enumerate}
		\item $F\cap G_{\a}=\varnothing$
		\item For every minimal generator $x^{\b}$ of $I$, there exists an index $i\notin F\cup G_{\a}$ with $b_{i}>a_{i}$.
	\end{enumerate}
	It is important to note that $\Delta_{\a}(I)$ may be either the empty set or $\{\varnothing\}$ and its vertex set may be a proper subset of $[n]$.
	
	The \textit{Takayama's formula} (see \cite{MT2011, T2005}), which generalizes Hochster's formula, is stated as follows. 
	
	\begin{theorem}[Takayama] \label{thm.Takayama}
		Let $I$ be a monomial ideal in $S$. For each $i\in\mathbb{Z}$ and $\a \in \mathbb{Z}^{n}$, there is an isomorphism of $\kk$-vector spaces 
		$$H_{\mathfrak{m}}^{i}(S/I)_{\a}\cong\widetilde{H}^{i-|G_{\a}|-1}(\Delta_{\a}(I);\kk),$$
		and $H_{\mathfrak{m}}^{i}(S/I)_{\a}=(0)$ if $G_{\a}\notin\Delta(\sqrt{I})$.
	\end{theorem} 
	
	Takayama's formula allows one to recover Reisner's criterion (see \cite{Reisner}) for Cohen-Macaulay monomial ideals.
	
	\begin{theorem}[Reisner's Criterion] \label{thm.Reisner}
		Let $I \subseteq S$ be a monomial ideal. The following are equivalent:
		\begin{enumerate}
			\item $S/I$ is Cohen-Macaulay of dimension $d$; and
			\item For every $\a \in \ZZ^n$,
			$$\widetilde{H}^j(\left(\Delta_\a(I);\kk\right) = 0 \text{ for all } j < d-|G_{\a}|-1.$$
		\end{enumerate}
	\end{theorem}
	
\subsection{Edge Ideals of Graphs and Other Conventions}

By a \emph{graph} over the vertex set $[n]$ we always mean a \emph{simple} (i.e., no loops nor multiple edges) and \emph{connected} graph with vertices $\{1, \dots, n\}$. We shall use $E(G)$ to denote the edge set of a graph $G$ (and sometimes use $V(G)$ to denote the vertex set of $G$ if it is not transparent from the context). For a vertex $x \in V(G)$ or, more generally, a subset $F \subseteq V(G)$, let $N(x)$ and $N(F)$ denote the sets of vertices that are adjacent to $x$ or to at least a vertex in $F$, respectively. Set $N[x] = N(x) \cup \{x\}$ and $N[F] = N(F) \cup F$.

The \emph{edge ideal} of a graph $G$ (originally defined in \cite{SVV1994}) over the vertices $[n]$ is given by
$$I(G) = \left(x_ix_j ~\big|~ \{i,j\} \in E(G)\right) \subseteq S = \kk[x_1, \dots, x_n].$$
For any vector $\a=(a_{1},\cdots,a_{n})\in\mathbb{Z}^{n}_{\ge 0}$, we write $x^\a$ for the monomial $x_1^{a_1} \cdots x_n^{a_n}$. For a subset $F \subseteq [n]$, set $x_F = \prod_{i \in F} x_i$. 

We use $\{\e_1, \dots, \e_n\}$ to denote the standard basis of $\RR^n$, and set $\e_F = \sum_{i \in F} \e_i$ for any subset $F \subseteq [n]$.
By convention, the dimension of the zero module is $-1$. 
	
	
\section{Relative Degree Complexes and Hochster--Takayama Formula} \label{sec.HT}
	
	In this section, we establish the Hochster--Takayama formula for monomial ideal quotients and corresponding relative simplicial complexes. We start by identifying the relative version of degree complexes.

\begin{lemma}\label{Key0}
	If $J\subseteq I$ be monomial ideals in $S$. Then $(\Delta_\a(J),\Delta_\a(I))$ is a relative simplicial complex for any $\a\in\ZZ^n$.
\end{lemma}

\begin{proof} For any $\a\in\NN^n$, since $J\subseteq I$, we have $\sqrt{J: x^\a}\subseteq \sqrt{I: x^\a}$. By \cite[Lemma 2.19]{M+2022}, $(\Delta_\a(J),\Delta_\a(I))$ is a relative simplicial complex. For $\a\in\ZZ^n\setminus \NN^n$, let $F=\{i~|~a_i<0\}$ and put $\a^+\in\NN^n$ for $(\a^+)_i=\a_i$ if $i\not\in F$ and $(\a^+)_i=0$, otherwise. One can see that $\Delta_\a(I)=\link_{\Delta_{\a^+}(I)}F$ and $\Delta_\a(J)=\link_{\Delta_{\a^+}(J)}F$. This implies that $(\Delta_\a(J),\Delta_\a(I))$ is a relative simplicial complex, since $(\Delta_{\a^+}(J),\Delta_{\a^+}(I))$ is a relative simplicial complex.
\end{proof}
	
\medskip

\noindent\textbf{Relative Hochster--Takayama formula} Our first main result, the Hochster--Takayama formula for monomial ideal quotients, is stated as follows.

\begin{theorem}[Relative Hochster--Takayama]\label{HochsterTakayama}
	Let $J\subsetneq I$ be monomial ideals in $S$. For each $i\in\ZZ$ and $\a\in\ZZ^n$, there is an isomorphism of $\kk$-vector spaces
	$$H_{\mm}^{i}(I/J)_{\a}\cong \widetilde{H}^{i-|G_\a|-1}((\Delta_\a(J),\Delta_\a(I)); \kk),$$
	and $H_{\mm}^{i}(I/J)_{\a}=(0)$ if $G_\a \not\in \Delta(\sqrt{J})$.
\end{theorem}

\begin{proof}
	Let $\Check{\bf C}^\bullet=\Check{\bf C}^\bullet(x_1,\dots,x_n)$ be the $\ZZ^n$-graded \v{C}ech complex with respect to
	$\mm=(x_1,\dots,x_n)$. 
	Note that, for any $\ZZ^n$-graded $S$-module $N$,
	\[
	H^i_{\mm}(N)\cong H^i(\Check{\bf C}^\bullet\otimes_S N)
	\quad\text{and}\quad
	H^i_{\mm}(N)_\a\cong H^i\!\big((\Check{\bf C}^\bullet\otimes_S N)_\a\big).
	\]
	
	By Takayama's work \cite{T2005}, for any monomial ideal $K \subseteq S$, we have the following isomorphism of cochain complexes:
	\begin{align}
	(\Check{\bf C}^\bullet\otimes_S S/K)_\a[|G_\a|+1]\ \cong\ \widetilde C^\bullet(\Delta_\a(K);\kk).\label{eq.CheckComp}
	\end{align}

	Consider the short exact sequence of $\ZZ^n$-graded modules
	\[
	0\to I/J\to S/J\to S/I\to 0.
	\]
	Localization is exact, so tensoring with $\Check{\bf C}^\bullet$ yields a short exact sequence of cochain complexes,
	and taking the $\a$-graded strand preserves exactness:
	\[
	0\to (\Check{\bf C}^\bullet\otimes_S I/J)_\a \to (\Check{\bf C}^\bullet\otimes_S S/J)_\a
	\to (\Check{\bf C}^\bullet\otimes_S S/I)_\a \to 0.
	\]
	Furthermore, since $J\subseteq I$, we have $\Delta_\a(I)\subseteq \Delta_\a(J)$. Thus, the map
	$(\Check{\bf C}^\bullet\otimes_S S/J)_\a\to(\Check{\bf C}^\bullet\otimes_S S/I)_\a$
	is the restriction map induced by $\Delta_\a(I)\subseteq\Delta_\a(J)$.
	Together with (\ref{eq.CheckComp}), this implies that the kernel complex identifies (after shifting by $|G_\a|+1$) with the reduced relative cochain complex
	$\widetilde C^\bullet(\Delta_\a(J),\Delta_\a(I);\kk)$. Hence,
	\[
	(\Check{\bf C}^\bullet\otimes_S I/J)_\a[|G_\a|+1]\ \cong\ \widetilde C^\bullet(\Delta_\a(J),\Delta_\a(I);\kk).
	\]
	Taking cohomology gives
	\[
	H^i_{\mm}(I/J)_\a \ \cong\ \widetilde H^{\,i-|G_\a|-1}\!\big((\Delta_\a(J),\Delta_\a(I));\kk\big).
	\]
	
Finally, if $G_\a\notin\Delta(\sqrt J)$, then Takayama's vanishing result (see \cite{T2005}) implies $H^i_{\mm}(S/J)_\a=0$ for all $i$.
	Since $\sqrt J\subseteq \sqrt I$, also $G_\a\notin\Delta(\sqrt I)$, so $H^i_{\mm}(S/I)_\a=0$ for all $i$.
	The long exact sequence in local cohomology induced by
	$0\to I/J\to S/J\to S/I\to 0$ then yields $H^i_{\mm}(I/J)_\a=0$ for all $i$.
\end{proof}

The following Reisner-type implication is a direct consequence of the Relative Hochster--Takayama formula in Theorem \ref{HochsterTakayama} (see also \Cref{thm.Reisner}). 

\begin{corollary}\label{thm:relative-reisner-implication}
	Let $d=\dim(I/J)$. 
	Assume that for every $\mathbf a\in\ZZ^n$,
	\[
	\widetilde H^j\!\left(\Delta_{\mathbf a}(J),\Delta_{\mathbf a}(I);\kk\right)=0
	\quad\text{for all } j< d-|G_{\mathbf a}|-1.
	\]
	Then $I/J$ is Cohen-Macaulay of dimension $d$.
\end{corollary}

\begin{proof}
	Fix $i<d$ and $\mathbf a\in\ZZ^n$.
	By Theorem~\ref{HochsterTakayama},
	\[
	H^i_{\mm}(I/J)_{\mathbf a}
	\;\cong\;
	\widetilde H^{\,i-|G_{\mathbf a}|-1}
	\!\left(\Delta_{\mathbf a}(J),\Delta_{\mathbf a}(I);\kk\right).
	\]
	Since $i-|G_{\mathbf a}|-1<d-|G_{\mathbf a}|-1$, the right-hand side vanishes
	by hypothesis. Hence $H^i_{\mm}(I/J)=0$ for all $i<d$, and therefore $I/J$
	is Cohen-Macaulay.
\end{proof}

\medskip

Theorem \ref{HochsterTakayama} further gives the following rigidity result and an immediate corollary.

\begin{theorem}[Rigidity]\label{thm:rigidity}
	Suppose that $\dim I/J \ge 2$.
	If there exists $\mathbf a$ such that
	\[
	\widetilde H^0(\Delta_{\mathbf a}(J),\Delta_{\mathbf a}(I)) \neq 0,
	\]
	then $\depth I/J \le 1$.
\end{theorem}

\begin{proof}
	By the Relative Hochster--Takayama formula,
	$H^1_{\mm}(I/J)_{\mathbf a} \neq 0$, so $\depth I/J \le 1$.
\end{proof}

\begin{corollary}[$(S_2)$ collapse]\label{cor:S2}
	If $\dim(I/J)\ge 2$ and some
	$\widetilde H^0(\Delta_{\mathbf a}(J),\Delta_{\mathbf a}(I)) \neq 0$,
	then $I/J$ fails $(S_2)$.
\end{corollary}

\begin{proof}
	The assertion is a direct consequence of Theorem \ref{thm:rigidity}.
\end{proof}

\medskip

\noindent\textbf{Functoriality} In parallel to the Hochster--Takayama formula, we obtain the following functoriality result. We shall need a lemma.

\begin{lemma}\label{lem:gradedpiece-dim1}
	Let $J\subseteq I$ be monomial ideals in $S$, and set $M:=I/J$.
	Fix $F\subseteq[n]$ and $\a = (a_1, \dots, a_n) \in\ZZ^n$. Write $\a^-:=\sum_{i\in G_\a}(-a_i)\e_i\in\NN^n \text{ and } \a^+:=\a+\a^-\in\NN^n.$
	The following statements hold:
	\begin{enumerate}
		\item[(1)] If $F\not\supseteq G_\a$, then $(M_{x_F})_\a=0$.
		\item[(2)] If $F\supseteq G_\a$, then $(M_{x_F})_\a$ is either $0$ or a $1$-dimensional
		$k$-vector space. More precisely,
		\[
		(M_{x_F})_\a\neq 0
		\quad\Longleftrightarrow\quad
		x^{\a^+}\in I S_{x_F}\setminus J S_{x_F},
		\]
		and in this case $(M_{x_F})_\a$ is generated by the class of $\frac{x^{\a^+}}{x^{\a^-}}\ \in\ M_{x_F}.$
	\end{enumerate}
\end{lemma}

\begin{proof}
	We work in the $\ZZ^n$-graded localization $S_{x_F}=S[x_i^{-1}\mid i\in F]$.
	Observe that a Laurent monomial $x^\gamma$ lies in $S_{x_F}$ if and only if $\gamma_i\ge 0$ for all $i\notin F$.
	Hence
	\[
	(S_{x_F})_\a\neq 0 \quad\Longleftrightarrow\quad a_i\ge 0 \ \text{ for all } i\notin F
	\quad\Longleftrightarrow\quad G_\a\subseteq F.
	\]
	If $G_\a\subseteq F$, then $\supp(\a^-)\subseteq G_\a\subseteq F$, so $x^{\a^-}$ is a unit in $S_{x_F}$
	and the element $x^{\a^+}/x^{\a^-}$ is a Laurent monomial of degree $\a$.
	
	We now show that $(S_{x_F})_\a$ is $1$-dimensional over $k$, spanned by $x^{\a^+}/x^{\a^-}$.
	Let $x^\alpha/x^\beta$ be any Laurent monomial in $S_{x_F}$ of degree $\a$, where
	$\alpha,\beta\in\NN^n$ and $\supp(\beta)\subseteq F$. The degree condition means $\alpha-\beta=\a$.
	For each $i\in G_\a$ we have $a_i<0$, hence $\alpha_i=\beta_i+a_i\ge 0$ forces $\beta_i\ge -a_i$.
	Thus $\beta\ge \a^-$ componentwise. Consequently $\beta-\a^-\in\NN^n$, and then
	\[
	\alpha-\a^+=\alpha-(\a+\a^-)= (\beta+\a)-(\a+\a^-)=\beta-\a^-\in\NN^n.
	\]
	Set $\delta:=\beta-\a^-=\alpha-\a^+\in\NN^n$. Then
	\[
	\frac{x^\alpha}{x^\beta}
	=
	\frac{x^{\a^+}x^\delta}{x^{\a^-}x^\delta}
	=
	\frac{x^{\a^+}}{x^{\a^-}}.
	\]
	So there is \emph{exactly one} Laurent monomial in $S_{x_F}$ of degree $\a$.
	Therefore $(S_{x_F})_\a$ is $1$-dimensional over $k$, spanned by $x^{\a^+}/x^{\a^-}$.
	
	\smallskip
(1) If $F\not\supseteq G_\a$, then $(S_{x_F})_\a=0$ by our arguments above. Thus,
	$(I S_{x_F})_\a=(J S_{x_F})_\a=0$ and, therefore, $(M_{x_F})_\a=0$.
	
(2) Assume that $F\supseteq G_\a$. As we have shown, in this case, $(S_{x_F})_\a$ is $1$-dimensional and spanned by
	$x^{\a^+}/x^{\a^-}$. Since $I S_{x_F}$ and $J S_{x_F}$ are monomial ideals in $S_{x_F}$, they
	are $\ZZ^n$-graded, so $(I S_{x_F})_\a$ and $(J S_{x_F})_\a$ are $k$-subspaces of $(S_{x_F})_\a$.
	Because $(S_{x_F})_\a$ is $1$-dimensional, each of these graded pieces is either $0$ or all of
	$(S_{x_F})_\a$.
	
	Hence $(M_{x_F})_\a=(I S_{x_F}/J S_{x_F})_\a$ is either $0$ or $1$-dimensional, and it is nonzero
	if and only if
	\[
	\frac{x^{\a^+}}{x^{\a^-}}\in I S_{x_F}
	\quad\text{and}\quad
	\frac{x^{\a^+}}{x^{\a^-}}\notin J S_{x_F}.
	\]
	Since $x^{\a^-}$ is a unit in $S_{x_F}$, for any ideal $K$ we have
	\[
	\frac{x^{\a^+}}{x^{\a^-}}\in K S_{x_F}
	\quad\Longleftrightarrow\quad
	x^{\a^+}\in K S_{x_F}.
	\]
	Applying this with $K=I$ and $K=J$ yields
	\[
	(M_{x_F})_\a\neq 0
	\quad\Longleftrightarrow\quad
	x^{\a^+}\in I S_{x_F}\setminus J S_{x_F}.
	\]
	When this holds, the class of $x^{\a^+}/x^{\a^-}$ spans the (necessarily $1$-dimensional)
	vector space $(M_{x_F})_\a$, completing the proof.
\end{proof}

Functoriality for one ideal was proved in \cite{MN2011}. Our next result gives functoriality for relative degree complexes and local cohomology modules of monomial ideal quotients.

\begin{theorem}[Functoriality]\label{MainLem}
	Let $J\subseteq I$ be monomial ideals in $S$ and let $M = I/J$.
	Fix $i\in\mathbb Z_{\ge 0}$, and let $\a\in\mathbb Z^n$ and $\b\in\mathbb N^n$ be such that $G_{\a+\b} = G_\a$.
	Then multiplication by $\x^\b$ induces a commutative diagram
	\[
	\begin{tikzcd}
		H^i_{\mathfrak m}(M)_\a
		\arrow[r,"{\cdot \x^\b}"]
		\arrow[d,"\cong"']
		&
		H^i_{\mathfrak m}(M)_{\a+\b}
		\arrow[d,"\cong"]
		\\
		\widetilde H^{\,i-|G_\a|-1}\!
		\big(
		\Delta_\a(J),\,\Delta_\a(I)
		\big)
		\arrow[r]
		&
		\widetilde H^{\,i-|G_\a|-1}\!
		\big(
		\Delta_{\a+\b}(J),\,\Delta_{\a+\b}(I)
		\big),
	\end{tikzcd}
	\]
	where the vertical isomorphisms are given by the relative Hochster--Takayama
	formula, and the bottom horizontal map is induced by the natural inclusions
	\[
	\Delta_{\a+\b}(J)\ \subseteq\ \Delta_\a(J)
	\text{ and }
	\Delta_{\a+\b}(I)\ \subseteq\ \Delta_\a(I).
	\]
	In particular, under the hypothesis $G_{\a+\b}=G_\a$, the multiplication map
	\[
	\cdot \x^\b:\ H^i_{\mathfrak m}(M)_\a \longrightarrow H^i_{\mathfrak m}(M)_{\a+\b}
	\]
	is zero if and only if the induced map on relative reduced cohomology is zero.
\end{theorem}

\begin{proof}
	Let $\b\in\NN^n$ and $\a\in\ZZ^n$ satisfy $G_{\a+\b}=G_\a$. Set $G:=G_\a=G_{\a+\b}$.
	Multiplication by $x^\b$ defines a homogeneous $S$-linear map $M\to M$ of degree $\b$, so a morphism of
	$\ZZ^n$-graded \v{C}ech complexes
	\[
	\mu_\b^\bullet:\ \Check{\bf C}^\bullet\otimes_S M \longrightarrow \Check{\bf C}^\bullet\otimes_S M,
	\qquad
	u\longmapsto x^\b u,
	\]
	which commutes with the \v{C}ech differential because all localization maps are $S$-linear. Taking the $\a$-graded
	strand gives a cochain map
	\[
	(\mu_\b^\bullet)_\a:\ (\Check{\bf C}^\bullet\otimes_S M)_\a \longrightarrow (\Check{\bf C}^\bullet\otimes_S M)_{\a+\b},
	\]
	whose induced map on cohomology is exactly $\cdot x^\b:\ H^i_\mm(M)_\a\to H^i_\mm(M)_{\a+\b}$.
	
	For any monomial ideal $K \subseteq S$, we observe that $\Delta_{\a+\b}(K)\subseteq \Delta_\a(K)$.
	Indeed, let $E\subseteq[n]\setminus G$ and set $F:=G\cup E$. It follow from the definition of degree complexes that
	\[
	E\in\Delta_{\a}(K)\quad\Longleftrightarrow\quad x^{\a^+}\notin K S_{x_F}\qquad(\a\in\ZZ^n).
	\]
	If $E\in\Delta_{\a+\b}(K)$, then $x^{(\a+\b)^+}=x^{\a^++\b}\notin K S_{x_F}$. If $x^{\a^+}\in K S_{x_F}$, then multiplying
	by $x^\b\in S\subseteq S_{x_F}$ gives $x^{\a^++\b}\in K S_{x_F}$, a contradiction. Hence, $x^{\a^+}\notin K S_{x_F}$ and
	$E\in\Delta_\a(K)$. This proves the inclusions for $K=I,J$, hence an inclusion of pairs
	\[
	(\Delta_{\a+\b}(J),\Delta_{\a+\b}(I))\hookrightarrow (\Delta_\a(J),\Delta_\a(I)).
	\]
	Let $\iota^\ast$ denote the induced restriction map on reduced relative cohomology.
	
	Recall from the proof of Theorem~\ref{HochsterTakayama} that the isomorphism
	\[
	H^i_\mm(M)_\a\ \cong\ \widetilde H^{\,i-|G_\a|-1}\big(\Delta_\a(J),\Delta_\a(I);\kk\big)
	\]
	is induced by an identification of cochain complexes
	\[
	\Phi_\a:\ (\Check{\bf C}^\bullet\otimes_S M)_\a[|G_\a|+1]\ \xrightarrow{\ \cong\ }\
	\widetilde C^\bullet\big(\Delta_\a(J),\Delta_\a(I);\kk\big),
	\]
	whose natural basis elements are described as follows.
	Fix $E\subseteq[n]\setminus G_\a$ and set $F:=G_\a\cup E$. By Lemma~\ref{lem:gradedpiece-dim1},
	$(M_{x_F})_\a$ is $0$ or $1$-dimensional, and if it is nonzero then it is generated by the class
	\[
	u_E(\a):=\left[\frac{x^{\a^+}}{x^{\a^-}}\right]\in (M_{x_F})_\a.
	\]
	Moreover, $(M_{x_F})_\a\neq 0$ holds if and only if $E\in\Delta_\a(J)\setminus\Delta_\a(I)$.
	Under $\Phi_\a$, the generator $u_E(\a)$ corresponds to the standard basis cochain $\varphi_E$ of the relative
	cochain complex of $(\Delta_\a(J),\Delta_\a(I))$.
	
	Now, fix $E\subseteq[n]\setminus G$ and $F:=G\cup E$. Assume $(M_{x_F})_\a\neq 0$, so $E\in\Delta_\a(J)\setminus\Delta_\a(I)$
	and $u_E(\a)$ is defined. Because $G_{\a+\b}=G_\a$, we have $(\a+\b)^-=\a^-$ and $(\a+\b)^+=\a^++\b$.
	Thus in $M_{x_F}$,
	\[
	x^\b\cdot u_E(\a)
	=
	x^\b\cdot \left[\frac{x^{\a^+}}{x^{\a^-}}\right]
	=
	\left[\frac{x^{\a^++\b}}{x^{\a^-}}\right]
	=
	\left[\frac{x^{(\a+\b)^+}}{x^{(\a+\b)^-}}\right]
	=
	u_E(\a+\b).
	\]
	If $E\in\Delta_{\a+\b}(J)\setminus\Delta_{\a+\b}(I)$, then $(M_{x_F})_{\a+\b}\neq 0$ and $u_E(\a+\b)$ is the generator,
	so $(\mu_\b^\bullet)_\a$ sends $u_E(\a)$ to the generator corresponding to the same face $E$.
	If $E\notin\Delta_{\a+\b}(J)\setminus\Delta_{\a+\b}(I)$, then $(M_{x_F})_{\a+\b}=0$ and hence $(\mu_\b^\bullet)_\a$ sends
	$u_E(\a)$ to $0$ in degree $\a+\b$.
	
	On the simplicial side, the restriction map $\iota^\ast$ sends the basis cochain $\varphi_E$ to $\varphi_E$ if
	$E$ is a face of the smaller relative complex, and to $0$ otherwise. We conclude, therefore, that the diagram of cochain complexes
	\[
	\begin{tikzcd}
		(\Check{\bf C}^\bullet\otimes_S M)_\a[|G|+1]
		\arrow[r,"(\mu_\b^\bullet)_\a"]
		\arrow[d,"\Phi_\a"']
		&
		(\Check{\bf C}^\bullet\otimes_S M)_{\a+\b}[|G|+1]
		\arrow[d,"\Phi_{\a+\b}"]
		\\
		\widetilde C^\bullet(\Delta_\a(J),\Delta_\a(I);\kk)
		\arrow[r,"\iota^\ast"]
		&
		\widetilde C^\bullet(\Delta_{\a+\b}(J),\Delta_{\a+\b}(I);\kk)
	\end{tikzcd}
	\]
	commutes on basis elements, and so commutes.
	Passing to cohomology gives the desired commutative diagram.
\end{proof}


\section{Relative Complexes for Squarefree Monomial Ideals} \label{sec.monIdeals}

This section specifies to the situation where $J \subseteq I$ are squarefree monomial ideals. In this case, as in \cite{Stanley} and \cite{AS2016}, the quotient $I/J$ is associated to a relative simplicial complex $(\Delta,\Gamma)$ with $\Delta$ and $\Gamma$ being the Stanley-Reisner complexes of $I$ and $J$, respectively. We shall see that for $\a \in \{0,\pm 1\}^n$, relative degree complexes at $\a$ are restrictions of $(\Delta,\Gamma)$ and, furthermore, links and localizations are compatible.

	\begin{lemma}
	\label{lem:SR-Takayama-induced}
	Let $\Sigma$ be a simplicial complex on $[n]$ and let $\mathbf a\in\{0,\pm 1\}^n$. Set $F = [n] \setminus G_\a$. Then, there is a natural identification
	\[
	\Delta_{\mathbf a}(I_\Sigma)=\Sigma|_F,
	\]
	where $\Sigma|_F$ denotes the induced subcomplex of $\Sigma$ on the vertex set $F$.
\end{lemma}

\begin{proof}
	By definition, a face $E\subseteq F$ belongs to $\Delta_{\mathbf a}(I_\Sigma)$ if and only if
	$x^E\notin I_\Sigma$, equivalently $E\in\Sigma$, and $E$ avoids all vertices in $G_{\mathbf a}$.
	This is precisely the definition of the induced subcomplex $\Sigma|_F$.
\end{proof}

Applying Lemma~\ref{lem:SR-Takayama-induced} to both $\Gamma$ and $\Delta$ yields the
following immediate consequence.

\begin{corollary}
	\label{cor:SR-relative-induced}
	Let $\Gamma\subseteq\Delta$ be simplicial complexes on $[n]$ and let $\mathbf a\in\{0,\pm 1\}^n$.
	Then there is a natural identification of pairs of simplicial complexes
	\[
	\left(\Delta_{\mathbf a}(I_\Delta), \Delta_{\mathbf a}(I_\Gamma)\right)
	\cong
	\left(\Delta|_F, \Gamma|_F\right),
	\]
	where $F=[n]\setminus G_{\mathbf a}$.
\end{corollary}

Thus, in squarefree multidegrees, the relative Hochster--Takayama formula for
$I_\Gamma/I_\Delta$ is governed by the relative topology of induced subcomplexes.

Combining Corollary~\ref{cor:SR-relative-induced} with the relative Hochster--Takayama
formula yields the following explicit description of the multigraded local cohomology
of Stanley--Reisner quotients.

\begin{theorem}[Relative Hochster--Takayama formula for Stanley--Reisner quotients]
	\label{thm:SR-relative-HT}
	Let $\Gamma\subseteq\Delta$ be simplicial complexes on $[n]$ and let
	$M=I_\Gamma/I_\Delta$.
	For every squarefree multidegree $\mathbf a\in\{0,\pm 1\}^n$ with $F=[n]\setminus G_{\mathbf a}$,
	there is a natural isomorphism
	\[
	H^i_{\mathfrak m}(M)_{\mathbf a}
	\cong
	\widetilde H^{i-|G_{\mathbf a}|-1}\left(\Delta|_F, \Gamma|_F\right).
	\]
\end{theorem}

\begin{proof}
	By the Relative Hochster--Takayama formula for monomial quotients,
	\[
	H^i_{\mathfrak m}(M)_{\mathbf a}
	\cong
	\widetilde H^{i-|G_{\mathbf a}|-1}
	\left(\Delta_{\mathbf a}(I_\Delta),\Delta_{\mathbf a}(I_\Gamma)\right).
	\]
	The claim now follows from Corollary~\ref{cor:SR-relative-induced}.
\end{proof}

The induced subcomplex description admits a link interpretation that is particularly
useful for localization and depth considerations.
Let $F\subseteq[n]$ and write $P_F=(x_i : i\notin F)$ for the corresponding monomial prime.
Localizing $I_\Gamma/I_\Delta$ at $P_F$ corresponds to restricting to the induced pair
$(\Delta|_F,\Gamma|_F)$.

Moreover, for faces $E\subseteq F$, the links satisfy
\[
\link_{\Gamma|_F}(E)=\link_\Gamma(E)|_{F\setminus E},
\quad
\link_{\Delta|_F}(E)=\link_\Delta(E)|_{F\setminus E}.
\]
Thus the relative topology of induced subcomplexes may equivalently be studied through
relative links.

This identification allows us to translate algebraic properties of
$I_\Gamma/I_\Delta$ such as support, dimension, depth, and sequential Cohen--Macaulayness
into linkwise topological conditions on the pair $(\Delta,\Gamma)$.


\section{Cohen-Macaulay and Generalized Cohen-Macaulay Properties} \label{sec.gCMB}

In this section, we obtain the Relative Reisner Criterion for Cohen-Macaulay monomial ideal quotients. We also characterize generalized Cohen-Macaulay monomial ideal quotients.

\medskip

\noindent\textbf{Relative Reisner Criterion.} Observe first that Corollary~\ref{thm:relative-reisner-implication} gives a uniform vanishing condition on the relative degree complexes that forces Cohen-Macaulayness of the monomial ideal quotient.
In order to obtain a genuine Relative Reisner criterion, one needs to
understand which multidegrees $\mathbf a$ can actually contribute to
local cohomology. The next lemma provides the key technical input for this purpose.

\begin{lemma}\label{lem:relevance-bound}
	Let $d=\dim(I/J)$. 
	If
	\[
	H^i_{\mm}(I/J)_{\mathbf a}\neq 0
	\quad\text{for some } i\in\ZZ,
	\]
	then
	\[
	\dim\!\left(\Delta_{\mathbf a}(J),\Delta_{\mathbf a}(I)\right)
	\;\le\;
	d-|G_{\mathbf a}|-1.
	\]
\end{lemma}

\begin{proof}
	Write $F=G_{\mathbf a}$ and let $\mathbf a^+\in\NN^n$ be obtained from
	$\mathbf a$ by replacing negative entries by $0$.
	It can be seen that
	\[
	\Delta_{\mathbf a}(J)=\link_{\Delta_{\mathbf a^+}(J)}(F),
	\qquad
	\Delta_{\mathbf a}(I)=\link_{\Delta_{\mathbf a^+}(I)}(F),
	\]
	and hence
	\[
	\dim\!\left(\Delta_{\mathbf a}(J),\Delta_{\mathbf a}(I)\right)
	\le \dim\Delta_{\mathbf a^+}(J)-|F|.
	\tag{$\ast$}
	\]
	
	Since $H^i_{\mm}(I/J)_{\mathbf a}\neq 0$, there exists a nonzero element
	of multidegree $\mathbf a$ in some localization $(I/J)_{x_F}$.
	Equivalently, there exist $\mathbf u\in\NN^n$ and $t\gg 0$ such that
	\[
	\mathbf u-t\mathbf 1_F=\mathbf a
	\quad\text{and}\quad
	x^{\mathbf u}\in I S_{x_F}\setminus J S_{x_F}.
	\]
	Clearing denominators, there exists $N\gg0$ such that
	\[
	x^{\mathbf c}:=x_F^N x^{\mathbf u}\in I\setminus J,
	\]
	and $\mathbf c\ge \mathbf a^+$ coordinatewise.
	
	Since $x^{\mathbf c}$ is divisible by $x^{\mathbf a^+}$, we have $(J:x^{\mathbf c})\subseteq (J:x^{\mathbf a^+}).$
	That is, $\Delta_{\mathbf a^+}(J)\subseteq \Delta_{\mathbf c}(J).$
	Therefore,
	\[
	\dim\Delta_{\mathbf a^+}(J)\le \dim\Delta_{\mathbf c}(J).
	\]
	
	Because $x^{\mathbf c}\in I\setminus J$, Lemma~\ref{lem:relevance-bound} yields $\dim\Delta_{\mathbf c}(J)\le d-1.$
	Combining this with $(\ast)$ gives
	\[
	\dim\!\left(\Delta_{\mathbf a}(J),\Delta_{\mathbf a}(I)\right)
	\le (d-1)-|F|=d-|G_{\mathbf a}|-1,
	\]
	as desired.
\end{proof}

We can now formulate a precise Reisner characterization.

\begin{theorem}[Relative Reisner Criterion]\label{thm:relative-reisner}
	Let $d=\dim(I/J)$. The following are equivalent:
	\begin{enumerate}
		\item $I/J$ is Cohen-Macaulay.
		\item For every $\mathbf a\in\ZZ^n$ and every face
		$F\in\Delta_{\mathbf a}(J)$,
		\[
		\widetilde H^j\!\left((
		\link_{\Delta_{\mathbf a}(J)}(F),
		\link_{\Delta_{\mathbf a}(I)}(F));\kk
		\right)=0
		\quad\text{for all } j< d-|G_{\mathbf a}|-|F|-1.
		\]
	\end{enumerate}
\end{theorem}

\begin{proof}
	$(1)\Longrightarrow(2)$.
	Assume that $I/J$ is Cohen-Macaulay of dimension $d$.
	Fix $\mathbf a$ and $F\in\Delta_{\mathbf a}(J)$.
	Let $\mathbf b\in\ZZ^n$ be defined by
	\[
	b_i=
	\begin{cases}
		-1,& i\in F,\\
		a_i,& i\notin F.
	\end{cases}
	\]
	Then $G_{\mathbf b}=G_{\mathbf a}\cup F$ and
	\[
	\Delta_{\mathbf b}(J)=\link_{\Delta_{\mathbf a}(J)}(F),
	\qquad
	\Delta_{\mathbf b}(I)=\link_{\Delta_{\mathbf a}(I)}(F).
	\]
	Since $I/J$ is Cohen-Macaulay, $H^i_{\mm}(I/J)=0$ for all $i<d$.
	Applying Theorem~\ref{HochsterTakayama} to $\mathbf b$ yields
	\[
	\widetilde H^{\,i-|G_{\mathbf b}|-1}
	\!\left(
	\link_{\Delta_{\mathbf a}(J)}(F),
	\link_{\Delta_{\mathbf a}(I)}(F)
	\right)=0
	\quad\text{for all } i<d,
	\]
	which is equivalent to the stated vanishing condition.
	
	\smallskip
	$(2)\Longrightarrow(1)$.
	Taking $F=\varnothing$ in (2) gives
	\[
	\widetilde H^j\!\left(\Delta_{\mathbf a}(J),\Delta_{\mathbf a}(I)\right)=0
	\quad\text{for all } j< d-|G_{\mathbf a}|-1.
	\]
	By Corollary~\ref{thm:relative-reisner-implication}, this implies that
	$I/J$ is Cohen-Macaulay.
\end{proof}


\medskip

\noindent\textbf{Generalized Cohen-Macaulay monomial ideal quotients.} 
Recall that $M$ is \emph{generalized Cohen-Macaulay} if $H^i_{\mm}(M)$ has finite length for all $i<d$. The following result characterizes generalized Cohen-Macaulay monomial ideal quotients by the vanishing of reduced homology of the corresponding relative degree complexes.

\begin{theorem}[Generalized Cohen-Macaulay property]\label{thm:gCM-criterion}
	The following are equivalent:
	\begin{enumerate}
		\item $I/J$ is generalized Cohen-Macaulay.
		\item For every $i<d$ there exists an integer $N(i)$ such that for all $\mathbf{a} = (a_1, \dots, a_n) \in\NN^n$
		with $\max(a_j) \ge N(i)$,
		\[
		\widetilde H^{\,i-1}\!\left((\Delta_{\mathbf a}(J),\Delta_{\mathbf a}(I));\kk\right)=0.
		\]
	\end{enumerate}
\end{theorem}

\begin{proof}
	Fix $i<d$. Since $H^i_{\mm}(I/J)$ is $\ZZ^n$-graded, it has finite length if and only if
	there exists $N(i)$ such that $H^i_{\mm}(I/J)_{\mathbf a}=0$ for all $\mathbf{a} \in\NN^n$
	with $\max(a_j)\ge N(i)$.
	For $\mathbf a\in\NN^n$ we have $G_{\mathbf a}=\varnothing$, so Theorem~\ref{HochsterTakayama} gives
	\[
	H^i_{\mm}(I/J)_{\mathbf a}\ \cong\ 
	\widetilde H^{\,i-1}\!\left((\Delta_{\mathbf a}(J),\Delta_{\mathbf a}(I));\kk\right).
	\]
	Thus, the finite length of $H^i_{\mm}(I/J)$ is equivalent to the stated eventual vanishing of
	the reduced relative cohomology. This holds for all $i<d$ if and only if $I/J$ is
	generalized Cohen-Macaulay.
\end{proof}

\medskip


%


\section{Symbolic Quotients} \label{sec.sym/sym}

In this section, we use Hochster--Takayama formula for relative degree complexes to study the Cohen-Macaulayness of symbolic quotients $I^{(t)}/I^{(t+1)}$, for $t \ge 1$, for a squarefree monomial ideal $I$. Our main result provides yet another interesting connection between this property to the matroidal structure of the corresponding simplicial complex, that is similar to the delicate characterization of the Cohen-Macaulay property of $I^{(t)}$, for $t \ge 1$, given in \cite{MT2011, V2011}.

We start with the following simple observation.

\begin{lemma}\label{dim-Key2}
	Let $\Delta$ be a simplicial complex and let $I=I_\Delta$ be its Stanley-Reisner ideal. Assume that $I \not= (0)$ and let $J$ be a monomial ideal such that $I^t \subseteq J \subseteq I^{(t)}$. Then, for any $t \ge 1$, 
	$$I^{(t+1)}:J=I.$$
	In particular, $\dim I^{(t)}/I^{(t+1)} = \dim S/I$.
\end{lemma}

\begin{proof} Clearly, $IJ\subseteq I.I^{(t)}\subseteq I^{(t+1)}$. Suppose, by contradiction, that there exists a monomial $x^\a\in(I^{(t+1)}:J)\setminus I$. Then, the squarefree monomial $x_G=\sqrt{x^\a}\notin I$, which implies that $G\in\Delta$. Let $F$ be a facet of $\Delta$ such that $G\subseteq F$. 
	
	Observe that $F\not= [n]$, so there exists $u\in [n]\setminus F$. It follows that $F\cup\{u\}\notin\Delta$, i.e., there exists a minimal generator of $I$, which is of the form $x_ux_{F'}$, for some $F'\subseteq F$. Therefore, 
	$$x^\a(x_ux_{F'})^t\in x^\a I^t\subseteq x^\a J\subseteq I^{(t+1)}=\bigcap_{H \in  \F(\Delta)} P_H^{t+1}.$$
	On the other hand, it can be seen that $\supp(x^\a(x_{F'})^t)\subseteq G\cup F'\subseteq F$. Thus, $x^\a(x_ux_{F'})^t \notin P_F^{t+1}$, which is a contradiction. The assertion is proved.
\end{proof}

The main result of this section is stated as follows. 

\begin{theorem}[Cohen-Macaulay symbolic quotients] \label{Main1}
	Let $\Delta$ be a simplicial complex on $[n]$ and let $I=I_{\Delta}$ be its Stanley--Reisner ideal.
	Then the following are equivalent:
	\begin{enumerate}
		\item[(1)] $I^{(t)}/I^{(t+1)}$ is Cohen-Macaulay for all $t\ge1$;
		\item[(2)] $I^{(t)}/I^{(t+1)}$ is Cohen-Macaulay for some $t\ge2$;
		\item[(3)] $\Delta$ is a matroid.
	\end{enumerate}
	Moreover, if $\dim S/I \ge2$, then the above are further equivalent to:
	\begin{enumerate}
		\item[(4)] $I^{(t)}/I^{(t+1)}$ satisfies Serre's condition $(S_2)$ for all $t\ge1$;
		\item[(5)] $I^{(t)}/I^{(t+1)}$ satisfies Serre's condition $(S_2)$ for some $t\ge2$.
	\end{enumerate}
	If $\dim S/I \le1$, then $I^{(t)}/I^{(t+1)}$ satisfies $(S_2)$ trivially for all $t\ge1$.
\end{theorem}

\begin{proof} It is clear that $(1)\Longrightarrow (2)$. Assume $(3)$ is true. By \cite{MT2011, V2011}, $S/I^{(t)}$ is Cohen-Macaulay for any $t\ge 1$. This, together with Lemma \ref{dim-Key2} and by considering the short exact sequence
	$$0\to  I^{(t)}/I^{(t+1)} \to S/I^{(t+1)}\to S/I^{(t)}\to 0$$
and its associated long exact sequence of local cohomology, implies that $I^{(t)}/I^{(t+1)}$ is a Cohen-Macaulay module for all $t\ge 1$. That is, $(3) \Longrightarrow (1)$. 
	
	We shall prove $(2)\Longrightarrow (3)$. Suppose that $I^{(t)}/I^{(t+1)}$ is Cohen-Macaulay for some $t\ge 2$ and $\Delta$ is not a matroid. Then, after a reindexing if necessary, there exist $F, G\in \Delta$ such that $F\setminus G=\{1\}$, $G\setminus F=\{2,3\}$ and $F\cup\{2\}\notin\Delta, F\cup\{3\}\notin\Delta$.
	
	Let $L = \link_{\Delta} \{1\} \cap \link_{\Delta} \{2,3\}$.
	Note that $F \cap G \in L$, so we can take 
	a facet $U$ of $L$ such that $F \cap G \subseteq U$.
	Let $\a = t\e_1+\e_2 + \e_3 -\e_{U}$.
	In particular, $\a^+ = t\e_1+\e_2 + \e_3$ and $G_{\a} = U$.
	It can be verified that
	$$\Delta_{\a^{+}}(I^{(t+1)}) = \starop_{\Delta} \{1\} \cup \starop_{\Delta}\{2,3\}$$
	and 
	$$\Delta_{\a^{+}}(I^{(t)}) \text{ is a cone over } \{1\}.$$
	Moreover, we get 
	\begin{align*}
		\Delta_{\a}(I^{(t+1)}) & = \link_{\Delta_{\a^+}(I^{(t+1)})} G_\a \\
		& =  \link_{\starop_{\Delta} \{1\}} U  \cup \link_{\starop_{\Delta} \{2,3\}}U. 
	\end{align*}

Observe first that, since $\{1\} \in \link_{\starop_{\Delta} \{1\}} U$ and 
	$\{2,3\} \in \link_{\starop_{\Delta} \{2,3\}} U$, both these links are not the empty nor the void complexes.
We shall further show that 
	$$\link_{\starop_{\Delta} \{1\}} U  \cap \link_{\starop_{\Delta} \{2,3\}}U 
	= \{ \varnothing \}.$$
	Indeed, suppose on the contrary that there exists a vertex $x \in [n]$ that
	belongs to the intersection $\link_{\starop_{\Delta} \{1\}} U  \cap \link_{\starop_{\Delta} \{2,3\}}U$.
	Then, 
	$x \not\in U$, $U \cup \{x\} \in \starop_{\Delta}\{1\}$ and
	$U \cup \{x\} \in \starop_{\Delta}\{2,3\}$.
	In other words,
	$x \not\in U$, $U \cup \{x,1\} \in \Delta$ and $U \cup \{x,2,3\} \in \Delta$.
	Note that $x$ must be different from $1,2,3$ --- for if $x=2$ then $U \cup \{1,2\} \in \Delta$, whence 
	$(F \cap G) \cup \{1,2\} = F\cup \{2\} \in \Delta$, which is a contradiction --- similar arguments for the case where $x =1$ or $x=3$.
	Thus, we get $U \cup \{x\} \in \link_{\Delta} \{1\}$ and
	$U \cup \{x\} \in \link_{\Delta} \{2,3\}$. Consequently, we have
	$U \cup \{x\} \in L$, which contradicts to the maximality of $U$ in $L$.
	Hence, $\link_{\starop_{\Delta} \{1\}} U \cap \link_{\starop_{\Delta} \{2,3\}}U 
	= \{ \varnothing \}$. 
	
We have just shown that $\Delta_{\a}(I^{(t+1)})$ is the disjoint union of non-empty simplicial complexes $\link_{\starop_{\Delta} \{1\}} U$ and $\link_{\starop_{\Delta} \{2,3\}}U$. Let $r= |U| + 1$.
	By Theorem \ref{HochsterTakayama}, we have 
	\begin{align*}
		H_{\mm}^{r}( I^{(t)}/I^{(t+1)})_{\a}&\cong \widetilde{H}^{r-|G_\a|-1}(\Delta_\a(I^{(t+1)}),\Delta_\a(I^{(t)}))\\
		&=\widetilde{H}^{r-|U|-1}(\link_{\Delta_{\a^+}(I^{(t+1)})}U, \link_{\Delta_{\a^+}(I^{(t)})}U)\\
		&=\widetilde{H}^{0}(\link_{\starop_{\Delta} \{1\}} U  \cup \link_{\starop_{\Delta} \{2,3\}}U, \text{ the cone over } \{1\})\\
		&\cong \widetilde{H}^{0}(\link_{\starop_{\Delta} \{1\}} U  \cup \link_{\starop_{\Delta} \{2,3\}}U) \ne (0).
	\end{align*}
	Here, the last isomorphism follows from the standard long exact sequence of cohomologies of a pair of simplicial complexes and their relative complex.

Observe finally that $|U| \leq \dim(\link_{\Delta}\{2,3\})+1 \leq \dim \Delta-1$. Therefore, 
	$$r = |U| + 1 \leq \dim \Delta < \dim (S/I)=\dim(I^{(t)}/I^{(t+1)}),$$ 
	where the last equality is due to Lemma \ref{dim-Key2}. This implies that $I^{(t)}/I^{(t+1)}$ is not Cohen-Macaulay, a contradiction. We have established the equivalence between (1), (2) and (3).

Suppose now that $\dim S/I \ge 2$. By Lemma \ref{dim-Key2}, $\dim(I^{(t)}/I^{(t+1)}) = \dim S/I \ge 2$ for all $t \ge 1$. Clearly, (1)$\Longrightarrow$(4)$\Longrightarrow$(5) since Cohen-Macaulay implies $(S_2)$.
	It remains to prove (5)$\Longrightarrow$(3).
	If $\Delta$ is not a matroid, then the same witness multidegree $\a$ constructed in the proof of (2)$\Longrightarrow$(3)
	produces $\widetilde H^0\bigl((\Delta_{\a}(I^{(t+1)}),\Delta_{\a}(I^{(t)}))\bigr) \neq 0$ for $t \ge 2$.
	Hence, $H^1_{\mm}(I^{(t)}/I^{(t+1)}) \neq 0$, so $\depth(I^{(t)}/I^{(t+1)}) \le 1$. 
	Since $d\ge 2$, this contradicts $(S_2)$ at $\mm$.
	Thus $\Delta$ must be a matroid.
\end{proof}

\begin{example}
	Consider $I=(x_1x_2,x_2x_3,x_3x_4,x_4x_5,x_5x_1)$ in $S=\kk[x_1,x_2,x_3,x_4,x_5]$. Then, $I$ is the Stanley-Reisner ideal of a 5-cycle, which is not a matroid. In this example, $I^{(t)}/I^{(t+1)}$ is Cohen-Macaulay if and only if $t=1$.
\end{example}

It is desirable to see if a similar characterization to that of Theorem \ref{Main1} holds for monomial ideals in general.

\begin{prob} Characterize the Cohen-Macaulayness of $I^{(t)}/I^{(t+1)}$, for all $t \ge 1$, when $I$ is a monomial ideal.
\end{prob}



For \emph{pure} simplicial complexes generalized Cohen-Macaulay symbolic quotients admit a similar characterization as that of Cohen-Macaulay symbolic quotients, as illustrated in the following remark.

\begin{remark}[Generalized Cohen-Macaulay symbolic quotients] \label{rmk:gCMandB}
	Fix $t \ge 2$ and set $M=I_\Delta^{(t)}/I_\Delta^{(t+1)}$.  Assume that $\dim S/I_\Delta \ge 2$. By Lemma \ref{dim-Key2}, $\dim M = \dim S/I_\Delta \ge 2$.
	
	Observe that, by definition, if $M$ is generalized Cohen-Macaulay then $M$ is Cohen-Macaulay on the punctured spectrum.
	Localizing at $x_F=\prod_{i\in F}x_i$ for a nonempty face $F \in \Delta$ and using the fact that symbolic powers
	commute with localization (since $I_\Delta$ is radical), we obtain
	\[
	M_{x_F}\ \cong\ \kk[x_i^{\pm1}\mid i\in F]\otimes_\kk
	\Big(I_{\link_\Delta(F)}^{(t)}/I_{\link_\Delta(F)}^{(t+1)}\Big).
	\]
	Thus, $M_{x_F}$ is Cohen-Macaulay for every $F \neq \varnothing$. 
	By the Cohen-Macaulay classification in \Cref{Main1}, this forces $\link_\Delta(F)$ to be a matroid for every nonempty face $F \in \Delta$. In particular, $\Delta$ is \emph{locally matroidal}.  
	
	On the other hand, it can be seen from the standard short exact sequence that $\Ass(M) \subseteq \Ass(S/I_\Delta^{(t)}) \cup \Ass(S/I_\Delta^{(t+1)}) = \Min(I_\Delta)$. Thus, if $\Delta$ is pure then $M$ is equidimensional. Particularly, if, in addition, $M$ is Cohen-Macaulay on the punctured spectrum, then $M$ is generalized Cohen-Macaulay. 
	
	Hence, we have demonstrated that, under the hypotheses that $\Delta$ is pure, $t \ge 2$, and $\dim S/I_\Delta \ge 2$,
	$$I_\Delta^{(t)}/I_\Delta^{(t+1)}\ \text{is generalized CM} \Longleftrightarrow \Delta \text{ is locally matroidal.}$$
	In other words, if $\Delta$ is pure and of dimension at least 1, then the following are equivalent:
	\begin{enumerate}
		\item $I_\Delta^{(t)}/I_\Delta^{(t+1)}$ is generalized Cohen-Macaulay for all $t \ge 1$; 
		\item $I_\Delta^{(t)}/I_\Delta^{(t+1)}$ is generalized Cohen-Macaulay for some $t \ge 2$; 
		\item $\Delta$ is locally matroidal.
	\end{enumerate}
\end{remark}


\section{Dimension Stability and Symbolic-Ordinary Discrepancy Modules} \label{sec.dim}

In this section, we look at the dimension of monomial ideal quotients and its relationship with degree complexes. We also investigate the function $\dim I^{(t)}/I^t$, where $I$ is the edge ideal of a graph.

For a monomial ideal $I$, let $\G(I)$ denote the unique set of minimal monomial generators of $I$. For monomial ideals $J \subseteq I$, for simplicity, we shall write $x^\a \in \G(I) \setminus J$ to denote the set of minimal monomial generators of $I$ that are not in $J$. The following lemmas allow us to understand the dimension of monomial ideal quotients from corresponding degree complexes. 

\begin{lemma} \label{Lemma1} Let $J\subseteq I$ be monomial ideals. Then, 
\begin{alignat*}{2}
	\dim(I/J) & = \max\limits_{x^\a \in I \setminus J}\{\dim (S/\sqrt{J: x^\a})\} && = \max\limits_{x^\a \in I \setminus J}\{\dim(\Delta_\a(J))+1\} \\
	&= \max\limits_{x^\a \in \G(I) \setminus J}\{\dim (S/\sqrt{J: x^\a})\} && = 
	\max\limits_{x^\a \in \G(I) \setminus J}\{\dim(\Delta_\a(J))+1\}.
\end{alignat*}
\end{lemma}

\begin{proof} It can be seen that 
	$$\sqrt{\Ann(I/J)}=\sqrt{J:I}=\bigcap\limits_{x^\a \in I \setminus J} \sqrt{J:x^\a} = \bigcap\limits_{x^\a \in \G(I) \setminus J} \sqrt{J:x^\a}.$$
This implies that
$$\dim(I/J) = \max\limits_{x^\a \in I \setminus J}\{\dim (S/\sqrt{J: x^\a})\} = \max\limits_{x^\a \in \G(I) \setminus J}\{\dim (S/\sqrt{J: x^\a})\}.$$
It follows further that
	$S/\sqrt{\Ann(I/J)}$ is the Stanley-Reisner ring of 
	$$\bigcup_{x^\a\in I\setminus J} \Delta_\a(J)=\bigcup_{x^\a\in \G(I)\setminus J} \Delta_\a(J).$$
	Particularly, the assertion follows from the well-known relationship between the dimension of a Stanley-Reisner ring and that of its corresponding simplicial complex (see, for example, \cite{Stanley}).
\end{proof}

\begin{lemma} \label{Lemma2} Let $J \subseteq I$ be monomial ideals. If $I/J$ is Cohen-Macaulay of dimension $\alpha$, then $\Delta_\a(J)$ is Cohen-Macaulay of dimension $\alpha-1$ for any $x^\a\in I\setminus J$.
\end{lemma}

\begin{proof} Consider any $\a\in\NN^n$ such that $x^\a\in I\setminus J$. By Lemma \ref{Lemma1}, $\dim(\Delta_\a(J))\le \alpha-1$. 
	
If $\dim(\Delta_\a(J))=\alpha-1$ then, by Reisner's criterion (see, for example, \cite{Stanley}) for Cohen-Macaulay Stanley-Reisner rings, Lemma \ref{Key0}, and Theorem \ref{HochsterTakayama}, it can be seen that $\Delta_\a(J)$ is Cohen-Macaulay. 

Suppose now that $\dim(\Delta_\a(J))<\alpha-1$. 
For a fixed $F\in \Delta_\a(J)$, consider $\b\in\ZZ^n$ with $\b_i=\a_i$ if $i\notin F$ and $\b_i=-1$ otherwise. Then, $\link_{\Delta_\a(J)}F=\Delta_\b(J)$. 
	By Theorem \ref{HochsterTakayama}, for any $j\le \dim(\link_{\Delta_\a(J)}F)< \alpha-|F|-1$,
\begin{align} 
	\widetilde{H}^{j}(\link_{\Delta_\a(J)}F)=\widetilde{H}^{j}(\Delta_\b(J))=(0). \label{eq.10}
\end{align}

	Since $x^\a\notin J$, we have $\Delta_\a(J)\ne\varnothing$. Choose $F$ to be a facet of $\Delta_\a(J)$. Then, $\link_{\Delta_\a(J)}F=\{\varnothing\}$. It follows that $\widetilde{H}^{-1}(\link_{\Delta_\a(J)}F)\ne (0)$. This is a contradiction to (\ref{eq.10}), as $|F| \le \dim(\Delta_\a(J))+1 < \alpha$ and so $-1< \alpha-|F|-1$. The result is proved.
\end{proof}

For the remaining of this section, we consider monomial ideal quotients of the form $I^{(t)}/I^t$, for a monomial ideal $I \subseteq S$, and address the following question.

\begin{quest} \label{quest.dim}
	Which numerical function $f: \NN \rightarrow \ZZ$ can be the dimension function 
	$$t \mapsto \dim I^{(t)}/I^t?$$ 
\end{quest}

\begin{remark}
	\label{rmk.dim}
	It follows from \cite[Theorem 3.3 and Proposition 5.1]{HPV2008} that $\dim (I^{(t)}/I^t)$ is asymptotically a constant function. The question of interest is which asymptotically constant functions are such dimension functions of monomial ideals.
\end{remark}

We shall identify a class of monomial ideals, including edge ideals of graphs, for which the function $t \mapsto \dim I^{(t)}/I^t$ is an increasing function.

\begin{definition}[Herzog--Qureshi]
An ideal $I$ in a Noetherian ring is said to satisfy the \emph{Ratliff condition} if $I^{t+1} :I = I^t$ for all $t \ge 1$.
\end{definition}

It was shown in \cite{HQ2014} that if $I$ satisfies the Ratliff condition then $I$ has the (strong) persistence property on associate primes of its powers. We shall prove that for ideals satisfying Ratliff condition, the dimension function of interested in \ref{quest.dim} must be non-decreasing.

\begin{theorem}[Dimension function of symbolic discrepancies] \label{thm.dimIncrease} 
	Let $I$ be a monomial ideal that satisfies the Ratliff condition. For any $t \in \NN$, we have
	$$\dim(I^{(t)}/I^t)\le\dim(I^{(t+1)}/I^{t+1}).$$ 
\end{theorem}

\begin{proof} By Lemma \ref{Lemma1}, there exists $\a\in\NN^n$ such that 
	$x^\a\in \G(I^{(t)})\setminus I^t$ and $$\dim(I^{(t)}/I^t)=\dim(\Delta_\a(I^t))+1.$$ 
	Observe that, since $I^{t+1}:I=I^{t}$, we have
	$$I^t:x^\a=\bigcap_{e\in E(G)} I^{t+1}: (x^\a\cdot x_e).$$
	Thus, by the Dimension Theorem, there exists an edge $f\in G$ such that
	$$\dim(S/\sqrt{I^t:x^\a})=\dim\left(S\Big/\sqrt{I^{t+1}:(x^\a\cdot x_f)}\right).$$

	It is easy to see that $x^\a\cdot x_f \in I^{(t+1)}$ since $x^\a \in I^{(t)}$. Furthermore, since $x^\a \not\in I^t$, $\dim(S/\sqrt{I^{t+1} : (x^\a \cdot x_f)}) = \dim (S/\sqrt{I^t: x^\a}) \ge 0$, and so $x^\a \cdot x_f \not\in I^{t+1}$. Therefore, by Lemma \ref{Lemma1}, we have
	\begin{align*} 
		\dim(I^{(t)}/I^t) & = \dim(S/\sqrt{I^{t+1}: (x^\a \cdot x_f)}) \\
		& \le \max\{\dim(S/\sqrt{I^{t+1}: x^\b}) ~\big|~ x^\b \in I^{(t+1)} \setminus I^{t+1} \} \\
		& = \dim (I^{(t+1)}/I^{t+1}).
		\end{align*}
	The assertion is proved. 
\end{proof}

The proof of Theorem \ref{thm.dimIncrease} gives the following interesting corollary.

\begin{corollary}
	\label{cor.CM}
	Let $I$ be a monomial ideal satisfying the Ratliff condition. Suppose that $I^{(t)}/I^t$ is a nonzero Cohen-Macaulay module for some $t \in \NN$. Then, for any $s \le t$, either $I^{(s)}/I^s = 0$ or $\dim (I^{(s)}/I^s) = \dim(I^{(t)}/I^t).$
\end{corollary}

\begin{proof}
	Suppose that $I^{(s)} : I^s \not= \varnothing$. Observe that, by the Ratliff condition, $I^t : I^{t-s} = I^s$. Thus, for any $x^\a \in I^{(s)} \setminus I^s$, we have
	$$I^s : x^\a = \bigcap_{x^\b \in \G(I^{t-s})} I^t : (x^\a \cdot x^\b).$$
	On the other hand, by Lemmas \ref{Lemma1} and \ref{Lemma2}, for any $x^\b \in \G(I^{t-s})$ such that $x^\a \cdot x^\b \not\in I^t$ (clearly, $x^\a \cdot x^\b \in I^{(t)}$), we have
	$$\dim(S/\sqrt{I^t: (x^\a \cdot x^\b)}) = \dim(I^{(t)}/I^t).$$
	Therefore, 
	$$\dim (S/\sqrt{I^s : x^\a}) = \dim(I^{(t)}/I^t).$$
	This is true for any $x^\a \in I^{(s)} \setminus I^s$. Hence, by Lemma \ref{Lemma1}, $I^{(s)}/I^s$ has the same dimension as that of $I^{(t)}/I^t$.
\end{proof}

We shall now focus on the class of edge ideals of graphs and investigate the dimension function of its symbolic-ordinary discrepancy modules.


\begin{corollary} \label{dim-graph}
	Let $G$ be a graph and let $I = I(G)$ be its edge ideal. Then, the function $t \mapsto \dim I^{(t)}/I^t$ is an increasing function.
\end{corollary}

\begin{proof}
	It follows from \cite[Lemma 2.12]{MV2012} that $I(G)$ satisfies the Ratliff condition. The assertion is now a direct consequence of Theorem \ref{thm.dimIncrease}.
\end{proof}

For a graph $G$, let $t(G)$ denote the \emph{stabilization index} of $\dim (I(G)^{(t)}/I(G)^t)$; that is, $t(G)$ is the smallest integer $t_0$ such that $\dim (I(G)^{(t)}/I(G)^t) = \dim(I(G)^{(t_0)}/I(G)^{t_0})$ for all $t \ge t_0$. 

\begin{quest}\label{quest.stability}
	Let $c(G) \in \NN$ be such that $2c(G)+1$ is the maximum length of an induced odd cycle in $G$ ($c(G) = 0$ if $G$ is bipartite). Is it true that 
	$$t(G) \le c(G)+1?$$
\end{quest}

We shall give two instances for which $t(G)$ is explicitly determined and Question \ref{quest.stability} has affirmative answers. Recall that a graph is called \emph{unicyclic} if it has a unique cycle. Note also that if $G$ is a \emph{bipartite} graph, i.e., $G$ does not contain any odd cycles, then $I^{(t)} = I^t$ for all $t \in \NN$ (see, for example, \cite{HM2010, SVV1994}). Thus, we shall consider unicyclic graphs which contain odd cycles. Recall also that $\alpha(G)$ denotes the maximum size of an \emph{independent} set in $G$, i.e., a set of vertices which are pairwise nonadjacent. It is easy to see that $\dim(S/I(G)) = \alpha(G)$ for any graph $G$.

\begin{theorem}[Dimension function for symbolic-ordinary discrepancy of unicyclic graphs] \label{thm.unicyclic} Let $G$ be a unicyclic graph containing an odd cycle. Let $C_{2s+1}$ denote the unique odd cycle (of length $2s+1$, for $s \ge 1$) in $G$ and let $I = I(G)$. For any $t \ge s+1$, we have 
	$$\dim(I^{(t)}/I^{t})=\dim(I^{(s+1)}/I^{s+1})=\alpha(G[V\setminus N[C_{2s+1}]]).$$
\end{theorem}

\begin{proof} Using \cite[Theorem 3.4]{GHOS}, we have $I^{(s+1)}=I^{s+1}+(x^\a)$, where $x^\a=\prod\limits_{i\in C_{2s+1}}x_i$. By Lemma \ref{Lemma1},
	$$\dim(I^{(s+1)}/I^{s+1})=\dim(\Delta_\a(I^{s+1}))+1.$$
	
Consider any $t\ge s+1$ and write $t=k(s+1)+r$, for $0\le r\le s$. \cite[Theorem 3.4]{GHOS} again gives us that
	$$I^{(t)}=\sum_{i=0}^{k}I^{t-i(s+1)}(x^\a)^i.$$
	By Lemma \ref{Lemma1}, there exists $x^\b \in I^{t-i(s+1)}$, for some $1 \le i \le k$, such that
	$$\dim(I^{(t)}/I^{t})=\dim\left(S/\sqrt{I^t:x^\b(x^\a)^i}\right) = \max\limits_{x^\c \in \G(I^{t-i(s+1)})} \{\dim (S/\sqrt{I^t: x^\c(x^\a)^i})\}.$$
	
We shall show that it can be reduced to the case where $r=0$ and $i=k$. Indeed, by \cite[Lemma 2.12]{MV2012}, $I^{t}:I^{t-i(s+1)}=I^{i(s+1)}$, and so
	$I^{i(s+1)}:(x^\a)^i=\bigcap_{x^\c\in \G(I^{t-i(s+1)})} I^{t}:[(x^\a)^i\cdot x^\c].$
	Particularly, 
	$$\dim\left(S\Big/\sqrt{I^{i(s+1)}:(x^\a)^i}\right) = \max\limits_{x^\c \in \G(I^{t-i(s+1)})} \{\dim (S/\sqrt{I^t: (x^\a)^i \cdot x^\c})\} = \dim(S/\sqrt{I^t:x^\b(x^\a)^i}).$$ 
That is, we can consider $I^{(i(s+1))}/I^{i(s+1)}$ in place of $I^{(t)}/I^t$. Observe further that, for any $i \in \NN$,
	$$\sqrt{I^{i(s+1)}:(x^\a)^i}=I(G[V\setminus N[C_{2s+1}]])+(N[C_{2s+1}])$$
Hence, $\dim(S/\sqrt{I^{i(s+1)}:(x^\a)^i})=\dim(S/\sqrt{I^{s+1}:x^\a})$, and the desired statement is established.
\end{proof}

Recall that a \emph{perfect} graph is such that both the graph and its complement have no odd induced cycles of length $\ge 5$.

\begin{theorem}[Dimension function for symbolic-ordinary discrepancy of perfect graphs] \label{thm.perfect}
	Let $G$ be a perfect graph and let $I = I(G)$. Then, for any $t \ge 2$, we have
$$\dim(I^{(t)}/I^{t}) = \dim(I^{(2)}/I^{2}).$$
\end{theorem}

\begin{proof} By the celebrated Strong Perfect Graph Theorem, $G$ does not contain any induced odd cycle of length $\ge 5$. Thus, either $G$ is a bipartite graph or $G$ contains induced triangles. We may assume that $G$ is not bipartite, as the conclusion becomes trivial otherwise. We will prove the assertion for $t \ge 2$ by induction. The statement is vacuously true for $t=2$. Suppose that $t\ge 3$ and, by Lemma \ref{Lemma1}, $\dim(I^{(t)}/I^{t})=\dim(\Delta_\a(I^{t}))+1=\dim(S/\sqrt{I^t:x^\a})$ for some $x^\a\in G(I^{(t)})\setminus I^{t}$. 
	
By \cite[Theorem 3.10]{Su2008}, we have
	$$x^\a=\prod_{i=1}^r x_{V_i}$$
	where $V_i$'s are cliques of $G$ with $|V_i|\ge 2$, $\sum_1^r(|V_i|-1)=t$ and, since $x^\a \not\in I^t$, there exist some values $i$ such that $|V_i|\ge 3$.
	
Observe first that we may reduce to the case where $|V_i|\ge 3$ for all $i=1,\ldots,r$. Indeed, by reindexing if necessary, we may assume that $|V_i| \ge 3$ for $i = 1, \dots, s$ and $|V_i| = 2$ for $i = s+1, \dots, r$. Rewrite $x^\a=x^\b\cdot x^\c$ where $x^\b=\prod_{i=1}^s x_{V_i}$ and
	$x^\c=\prod_{i=s+1}^r x_{V_i}$. By \cite[Lemma 2.12]{MV2012}, we have
	$$\bigcap_{x^\b \in I^{(s)} \setminus I^s} I^{t-(r-s)}:x^\b=\bigcap_{x^\b \in I^{(s)}\setminus I^s} \bigcap_{x^\c\in I^{r-s}} I^t:(x^\b \cdot x^\c) = \bigcap_{x^\a \in I^{(t)} \setminus I^t} I^t: x^\a.$$
	Thus, we can replace $x^\a$ by $x^\b$.
	
We proceed by establishing the following claim.
	
\noindent \textbf{Claim:} $\sqrt{I^t:x^\a}=\bigcap_{i\in \supp\a}\sqrt{I^{t-1}:x^{\a-\e_i}}.$
	
\noindent\textit{Proof of the Claim:} Let $x_F$ be a minimal (squarefree) monomial generator of $\sqrt{I^t:x^\a}$. We may assume that $F$ is an independent set of $G$, for if $F$ is not an independent set then $x_F \in I \subseteq \sqrt{I^{t-1} : x^\a}$ obviously holds. Set $x^\a(F)=\prod_{i\notin N[F]}x_i^{a_i}$. By  \cite[Lemma 2.18]{MV2023}, we have
	$$\sum_{j\in N(F)}a_j+\ord_{I}(x^\a(F))\ge t.$$. 
	
	Fix $i\in\supp \a$, and let $\b=\a-\e_i$. 
	If $i\in N(F)$ then $\sum_{j\in N(F)}b_j=\sum_{j\in N(F)}a_j-1$ and $x^\a(F)=x^\b(F)$. 
	On the other hand, if $i\notin N(F)$ then $\sum_{j\in N(F)}b_j=\sum_{j\in N(F)}a_j$ and $\ord_I(x^\b(F))\ge \ord_I(x^\a(F))-1$. Then, in both of cases, we always have 
	$$\sum_{j\in N(F)} b_j + \ord_{I} (x^\b(F))\ge t-1.$$ Using \cite[Lemma 2.18]{MV2023} again, we get $x_F\in\sqrt{I^{t-1}:x^\b}$. 
	
	Conversely, we assume that $x_F$ is a squarefree monomial generator of $\bigcap_{i\in \supp\a}\sqrt{I^{t-1}:x^{\a-\e_i}}$ and $F$ is an independent set of $G$. Then, there exists a positive integer $\alpha$ such that
	$$x_F^\alpha \cdot x^{\a-\e_i}\in I^{t-1}\quad\text{ for all } i\in\supp\a.$$
	
	If $N(F)\cap\supp(\a)\not= \varnothing$, then take $j \in N(F) \cap \supp(\a)$. We have
	$$x_F^{\alpha+1}x^\a=(x_F^\alpha\cdot x^{\a-\e_j})\cdot(x_Fx_j)\in I^t.$$
	This implies that $x_F\in\sqrt{I^t:x^\a}$.
	
	Suppose that $N(F)\cap\supp(\a)=\varnothing$. In this case, since $F$ is an independent set of $G$, $x^{\a-\e_i}$ must belong to $I^{t-1}$ for all $i\in\supp(\a)$. In particular, $\deg(x^\a)\ge 2t-1$. On the other hand, we have
	$$\deg(x^\a)=\sum_{i=1}^r|V_i|=t+r.$$
	Therefore $r\ge t-1$. But $|V_i|\ge 3$ for all $i$, it implies $$t=\sum_1^r(|V_i|-1)\ge 2r\ge 2(t-1),$$ which is a contradiction. The claim is now proved. \qed
	
	We continue with the proof of Theorem \ref{thm.perfect}. It follows from the previous claim that $$\dim(S/\sqrt{I^t:x^\a})=\dim(S/\sqrt{I^{t-1}:x^{\a-\e_i}})$$ for some $i\in\supp\a$. Moreover, since $|V_i|\ge 3$, \cite[Theorem 3.10]{Su2008} implies that $x^{\a-\e_i}\in I^{(t-1)}$. Since $\dim(S/\sqrt{I^{t-1}:x^{\a-\e_i}}) = \dim (S/\sqrt{I^t: x^\a}) \ge 0$, we also have $x^{\a-\e_i}\notin I^{t-1}$. It follows that $\dim(I^{(t)}/I^t)\le\dim(I^{(t-1)}/I^{t-1})$. This, together with Lemma \ref{Lemma1}, gives 
	$$\dim (I^{(t)}/I^t) = \dim (I^{(t-1)}/I^{t-1}).$$
	By induction, $\dim(I^{(t)}/I^t)=\dim(I^{(2)}/I^{2})$ as desired.
\end{proof}

The following example shows that $\dim(I^{(t)}/I^t)$ can take 2 different values that are not $-1$, and these 2 values can be arbitrarily far apart. We do not know of any examples where $\dim(I^{(t)}/I^t)$ takes 3 different values that are not $-1$.

\begin{example} \label{ex.2values}
	For any $r \ge 1$, let $S = \kk[w_1, \cdots, w_5, x_1, \cdots x_{2+r}]$ and consider 
	$$I=(w_1w_2,w_2w_3,w_3w_4,w_4w_5,w_5w_1,w_1x_1,w_1x_2,x_1x_2,x_2x_3,\cdots,x_2x_{2+r}) \subseteq S.$$
	Then, 
	$$\dim(I^{(t)}/I^t) = \left\{ \begin{array}{rl} -1 & \text{if } t = 1 \\
		1 & \text{if } t = 2 \\
		r & \text{otherwise.} \end{array} \right.$$
\end{example}

\begin{prob}
	Compute or bound $t(G)$ in terms of combinatorial data of $G$.
\end{prob}


\section{Cohen-Macaulay Symbolic-Ordinary Discrepancy Modules} \label{sec.sym/ord}

We continue our investigation of monomial ideal quotients of the form $I^{(t)}/I^t$, where $I$ is the edge ideal of a graph, and examine their Cohen-Macaulay property. Throughout this section, $G$ denotes a graph on the vertex set $V = [n]$, and $I = I(G)$ is its edge ideal. We also assume that $G$ is not a bipartite graph, i.e., $G$ contains induced odd cycles.

\begin{lemma}
	\label{lem.NC}
	Let $C = C_{2t+1}$ be an induced odd cycle of minimum length in $G$ on the vertices $\{1, \dots, 2t+1\}$. Then, for any $s \ge 1$, we have $x_C \cdot (x_1x_2)^{s-1} \in I^{(t+s)} \setminus I^{t+s}$ and 
	$$\sqrt{I^{t+s} : (x_C \cdot (x_1x_2)^{s-1})} = I + (N[C]).$$
\end{lemma}

\begin{proof}
	Clearly, since $x_C \in I^{(t+1)}$, $x_C \cdot (x_1x_2)^{s-1} \in I^{(t+s)}$. By a degree comparison, it can also be seen that $x_C \cdot (x_1x_2)^{s-1} \not\in I^{t+s}$.
	
	It is furthermore easy to see that if $y \in N[C]$ then $yx_C (x_1x_2)^{s-1} \in I^{t+s}$, i.e., $y \in I^{t+s} : (x_C \cdot (x_1x_2)^{s-1}).$ Therefore,
	$$\sqrt{I^{t+s} : (x_C \cdot (x_1x_2)^{s-1})} \supseteq I + (N[C]).$$
	On the other hand, let $x_F$ be a minimal monomial generator of $\sqrt{I^{t+s}: (x_C \cdot (x_1x_2)^{s-1})}$. Then, either $x_F \in I$ or $F$ is an independent set of $G$. In the case that $F$ is an independent set, by \cite[Lemma 2.18]{MV2023}, we have
	$$\sum_{j \in N(F)} \ord_{x_j}(x_C \cdot (x_1x_2)^{s-1}) \ge (t+s) - \ord_I (x_C \cdot (x_1x_2)^{s-1}) = (t+s) - (t+s-1) = 1.$$
	Hence, $N(F) \cap C \not= \varnothing$, i.e., $F \cap N[C] \not= \varnothing$. In other words, $x_F \in (N[C])$, and we deduce that 
	$$\sqrt{I^{t+s} : (x_C \cdot (x_1x_2)^{s-1})} \subseteq I + (N[C]).$$
	The lemma is established.
\end{proof}

\begin{lemma}
	\label{lem.NCy}
	Assuming the same hypothesis as in Lemma \ref{lem.NC}. Suppose that $y \in N[C] \ C$ and, in particular, $y \in N[1]$. Then, for any $s \ge 2$, we have $x_C \cdot (x_1x_2)^{s-2}(x_1y) \in I^{(t+s)} \setminus I^{t+s}$ and
	$$\sqrt{I^{t+s} : (x_C \cdot (x_1x_2)^{s-2} (x_1y))} = I + (N[C]) + (N(y)).$$
\end{lemma}

\begin{proof}
	The proof is almost identical to that of Lemma \ref{lem.NC}. As before, we first observe that $x_C \cdot (x_1x_2)^{s-2}(x_1y) \in I^{(t+s)} \setminus I^{t+s}$. It is also easy to see that 
	$$\sqrt{I^{t+s} : (x_C \cdot (x_1x_2)^{s-2} (x_1y))} \supseteq I + (N[C]) + (N(y)).$$
	To see the reserve inclusion, we again consider a minimal squarefree monomial generator $x_F$ of $\sqrt{I^{t+s} : (x_C \cdot (x_1x_2)^{s-2} (x_1y))}$. We may further assume that $F$ is an independent set in $G$ and that $F \cap N[C] = \varnothing.$ By the same argument, making use of \cite[Lemma 2.18]{MV2023}, we then get that $F \cap N(y) \not= \varnothing$. This establishes the desired reverse inclusion and proves the lemma.
\end{proof}

\begin{lemma} \label{lem.NCyz}
	Assuming the same hypotheses as in Lemma \ref{lem.NCy}. Suppose that $z \in N(y) \setminus N[C]$. Then, for $s \ge 3$, we have $x_C \cdot (x_1x_2)^{s-3}(x_1y)(yz) \in I^{(t+s)} \setminus I^{t+s}$ and 
	$$\sqrt{I^{t+s} : (x_C \cdot (x_1x_2)^{s-3}(x_1y)(yz))} = I + (N[C]) + (N(y,z)).$$
\end{lemma}

\begin{proof}
	The proof follows in exactly the same of arguments as in Lemmas \ref{lem.NC} and \ref{lem.NCy}. The details are left to the interested reader.
\end{proof}

Recall from Remark \ref{rmk.dim} that $\dim (I^{(t)}/I^t)$ is asymptotically a constant. As before, let $t(G)$ denote the smallest integer starting from which $\dim (I(G)^{(t)}/I(G)^t)$ is a constant. Also, let $c(G) \in \NN$ be such that $2c(G)+1$ is the maximum length of an induced cycle in $G$ ($c(G) = 0$ if $G$ is a bipartite graph). Our last main result is stated as follows.

\begin{theorem}[Cohen-Macaulay symbolic-ordinary discrepancy modules] \label{thm.CMEdge}
	Let $G$ be a connected graph that is not bipartite and let $I = I(G)$ be its edge ideal.
	Set
	\[
	t_0:=\max\{t(G),c(G)\}+3.
	\]
	The following are equivalent:
	\begin{enumerate}
		\item[(1)] $I^{(t)}/I^{t}$ is Cohen-Macaulay for all $t \ge 1$;
		\item[(2)] $I^{(t)}/I^{t}$ is Cohen-Macaulay for some $t\ge t_0$;
		\item[(3)] For all $t \ge 1$, either $I^{(t)}/I^t = (0)$ or $\dim(I^{(t)}/I^t) = 0$;
		\item[(4)] For any induced odd cycle $C$ in $G$, we have $N[C] = [n]$.
	\end{enumerate}
\end{theorem}

\begin{proof} It is obviously that $(3) \Longrightarrow (1) \Longrightarrow (2)$. We shall prove $(2)\Longrightarrow (4) \Longrightarrow (3)$. 
	
$(2) \Longrightarrow (4)$. Suppose that $M=I^{(t)}/I^{t}$ is Cohen-Macaulay for some $t\ge \max\{t(G),c(G)\}+3$. Let $\alpha=\dim(M)$. By \cite[Theorem 4.1]{CMS2002} and \cite[Corollary 4.5]{LT2019}, $M \not= 0$, and so $\alpha \ge 0$. Using Lemma \ref{Lemma2}, we get that $\Delta_\a(I^{t})$ is Cohen-Macaulay of dimension $\alpha-1$ for any $x^\a\in I^{(t)}\setminus I^{t}$. It then follows from Lemma \ref{Lemma1} that $S/\sqrt{I^{t}: x^\a}$ is Cohen-Macaulay of dimension $\alpha$ for all $x^\a \in I^{(t)} \setminus I^{t}$.
	
Consider any induced odd cycle $C$ in $G$ and, without loss of generality, assume that the vertices of $C$ are $\{1,2,\ldots,2t'+1\}$, for some $t' \le c(G)$. Write $t = t'+s$, for some $s \ge 3$. Suppose, by contradiction, that $[n] \setminus N[C] \not= \varnothing$. Then, by the connectivity of $G$, there exist vertices $y,z$ of $G$ such that 
$y \in N[C] \setminus C$, $z \not\in N[C]$, and $yz \in G$. Without loss of generality, we may assume that $y \in N(x_1)$. Set
\begin{align*}
	x^\b & = x_C \cdot (x_1x_2)^{s-1} \\
	x^\c & = x_C \cdot (x_1x_2)^{s-3}(x_1y)(yz).
\end{align*}
Since $x_C \in I^{(t'+1)}$, both $x^\b$ and $x^\c$ are in $I^{(t'+s)}$. Moreover, by a degree comparison, it can be seen that both $x^\b$ and $x^\c$ are not in $I^{t'+s}$. Thus, by Lemmas \ref{lem.NC} and \ref{lem.NCyz}, we have
$$\dim\left(S\big/I + (N[C])\right) = \dim \left(S\big/I + (N[C]) + (N(y)) + (N(z))\right) = \alpha.$$

On the other hand, observe that if $F$ is an independent set of the graph associated to $I + (N[C]) + (N(y)) + (N(z))$ with maximum cardinality $\alpha$, then $z \not\in F$ and $F \cup \{z\}$ is an independent set in the graph associated to $I + (N[C])$. This implies that
$$\dim \left(S\big/I + (N[C])\right) \ge \alpha+1 > \alpha,$$
a contradiction. We have shown that $N[C] = [n]$ for any induced odd cycle $C$ in $G$.

$(4) \Longrightarrow (3)$. Suppose that $N[C] = [n]$ for any induced odd cycle $C$ in $G$. It follows from \cite[Theorem 4.4]{LT2019} that, for any $t \in \NN$, either $I^t$ has no embedded primes, or the embedded primes of $I^t$ are of the form $P_F = (x_i ~\big|~ x_i \in F)$, for some $F \subseteq V$, such that there exists $U \subseteq V$ for which $N[U] \subseteq F$ and the induced subgraph $G_U$ of $G$ over $U$ is a \emph{strongly non-bipartite} graph (under additional conditions). Since $G_U$ is a non-bipartite graph, it contains an induced odd cycle $C_1$. Particularly, $N[U] \supseteq N[C_1] = [n]$, which implies that $F = [n]$ and $P_F = \mm$. That is, for any $t \in \NN$, either $I^t = I^{(t)}$ or the only embedded prime of $I^t$ is $\mm$. 
In the latter case we have $I^t = I^{(t)} \cap \qq$, where $\sqrt{\qq} = \mm$. Therefore,
$$\dim (I^{(t)}/I^t) = \dim \dfrac{I^{(t)}}{I^{(t)} \cap \qq} = \dim \dfrac{I^{(t)} + \qq}{\qq}  = 0.$$
The result is proved.
\end{proof}

\begin{remark}\label{rmk:gCMQt}
	Let $G$ be a simple connected graph on $[n]$ and let $I=I(G)$ its edge ideal. Fix $t\ge1$ and set
	\[
	Q_t(G)\ :=\ I^{(t)}/I^t.
	\]
	Assume $\dim Q_t(G)\ge2$. Then $Q_t(G)$ is generalized Cohen-Macaulay if and only if it is Cohen-Macaulay on the punctured spectrum, i.e., if and only if $(Q_t(G))_{x_F}$ is Cohen-Macaulay for every nonempty monomial $x_F$.
	
	Equivalently, it suffices to test $x_F=\prod_{i\in F}x_i$ with $F$ a nonempty independent set of $G$ (i.e., $\varnothing \not= F\in \text{Ind}(G)$). In this case one has a natural isomorphism
	\[
	(Q_t(G))_{x_F}\ \cong\ \kk[x_i^{\pm1}:i\in F]\ \otimes_\kk\ Q_t\bigl(G-N_G[F]\bigr).
	\]
	Hence, $Q_t(G)$ is generalized Cohen-Macaulay if and only if $Q_t(G-N_G[F])$ is Cohen-Macaulay for every nonempty independent set $F$. By \Cref{thm.CMEdge}, this is equivalent to the condition that, for any independent set $F$ in $G$ and any odd induced cycle $C$ in $G - N_G[F]$, $N_G[C] \cup N_G[F] = [n]$.
\end{remark}

We end the paper with the following questions.

\begin{quest} Let $I$ be a monomial ideal.
	\begin{enumerate}
		\item Characterize the Cohen-Macaulayness of $I^{(t)}/I^t$, for all $t \ge 1$.
		\item If $I^{(t)}/I^t$ is Cohen-Macaulay for $t = u$ and $t = v$, for some $u < v$, then does it imply that $I^{(t)}/I^t$ is Cohen-Macaulay for all $u \le t \le v$?
		\item Which functions could be the depth/regularity function $\depth (I^{(t)}/I^t)$ (or $\reg (I^{(t)}/I^t$))?
	\end{enumerate} 
\end{quest}

\subsection*{Acknowledgements}

Part of this work was done while the authors were visiting Vietnam Institute for Advanced Study in Mathematics (VIASM).  We would like to thank VIASM for its hospitality and financial support.

\subsection*{Data Availability} Data sharing is not applicable to this article as no datasets were generated or analyzed during the current study.


\subsection*{Competing interests} The authors have no relevant financial or non-financial interests to disclose.



	\end{document}